\pdfoutput=1
\documentclass[11pt]{amsart}
\usepackage{pinlabel}
\usepackage{amsmath} 
\usepackage{cjhebrew}
\usepackage{graphicx} 
\usepackage{subcaption}
\usepackage{tikz-cd}
\usepackage[all,cmtip]{xy}
\setkeys{Gin}{keepaspectratio}
\usepackage{hyperref}
\usepackage{url}
\usepackage{bm}
\usepackage{mathabx}
\usepackage[left=1.25in,top=1in,right=1.25in,bottom=1in,head=.1in]{geometry}
\usepackage{xcolor}
\usepackage{tikz}
\usepackage{adjustbox}
\usepackage[normalem]{ulem} %
\usepackage{enumerate}
\usepackage{hyperref}

\usepackage{verbatim}
\newcommand{\items}{\begin{itemize}[leftmargin=25pt,rightmargin=15pt]
  \setlength\itemsep{2pt}}

\usepackage{hyperref}

\newcommand{\stopitems}{\end{itemize}}

\usepackage{fancyhdr}

\newtheorem{theorem}{Theorem}[section] 
\newtheorem*{theorem*}{Theorem}
\newtheorem{lemma}[theorem]{Lemma}
\newtheorem{conjecture}[theorem]{Conjecture}

\newtheorem*{conjecture*}{Conjecture}
\newtheorem*{question*}{Question}
\newtheorem*{lemma*}{Lemma}
\newtheorem{proposition}[theorem]{Proposition}
\newtheorem{corollary}[theorem]{Corollary}
\newtheorem*{corollary*}{Corollary}
\theoremstyle{definition}
\newtheorem{definition}[theorem]{Definition}

\newtheorem{remark}[theorem]{Remark}

\newtheorem{example}[theorem]{Example}

\newtheorem*{example*}{Example}
\newtheorem*{remark*}{Remark}
\newtheorem*{remarks*}{Remarks}
\newtheorem*{addenda*}{Addenda}
\newtheorem*{construction*}{Construction}

\newcommand{\fs}{\mathfrak{s}}

\DeclareMathOperator{\diff}{Diff}

\DeclareMathOperator{\SWF}{SWF}
\newcommand{\coul}{Coul}

\DeclareMathOperator{\redu}{red}

\DeclareMathOperator{\Bafu}{BF}

\newcommand{\F}{\mathbb F}

\newcommand{\cals}{\mathcal S}

\newcommand{\s}{\cals}

\renewcommand{\phi}{\varphi}

\DeclareMathOperator{\MCG}{MCG}
\DeclareMathOperator{\reduced}{red}
\DeclareMathOperator{\Pic}{Pic}

\newcommand{\coker}{\operatorname{coker}}

\newcommand{\unred}[1]{ \ignorespaces}  

\usepackage{stmaryrd}

\title[The monodromy diffeomorphism of weighted singularities]{The monodromy diffeomorphism of weighted singularities and Seiberg--Witten theory}

\author{Hokuto Konno}
\address{Graduate School of Mathematical Sciences, the University of Tokyo, 3-8-1 Komaba, Meguro, Tokyo 153-8914, Japan \\and\\
RIKEN iTHEMS, Wako, Saitama 351-0198, Japan}
\email{konno@ms.u-tokyo.ac.jp}

\author{Jianfeng Lin}
\address{Yau Mathematical Sciences Center, Tsinghua University, Beijing, 100871, China}
\email{linjian5477@mail.tsinghua.edu.cn}

\author{Anubhav Mukherjee}
\address{Department of Mathematics, Princeton University, Princeton, 08540, USA}
\email{anubhavmaths@princeton.edu}

\author{Juan Muñoz-Echániz}
\address{Simons Center for Geometry and Physics, State University of New York, Stony Brook, 11794, USA}
\email{jmunozechaniz@scgp.stonybrook.edu}

\begin{document}
\maketitle
\setlength{\headheight}{12.0pt}

\begin{abstract}

    We prove that the monodromy diffeomorphism of a complex $2$-dimensional isolated hypersurface singularity of weighted-homogeneous type has infinite order in the smooth mapping class group of the Milnor fiber, provided the singularity is not a rational double point. 
    This is a consequence of our main result: the boundary Dehn twist diffeomorphism of an indefinite symplectic filling of the canonical contact structure on a negatively-oriented Seifert-fibered rational homology $3$-sphere has infinite order in the smooth mapping class group. Our techniques make essential use of analogues of the contact invariant in the setting of $\mathbb{Z}/p$-equivariant Seiberg--Witten--Floer homology of $3$-manifolds.

\end{abstract}

\section{Introduction}




In this article we study the smooth mapping class group $\pi_0 ( \mathrm{Diff}(M, \partial ))$ of smooth compact $4$-manifolds admitting a symplectic structure with convex boundary. Our main result (Theorem \ref{thm: main}) establishes the infinite-order non-triviality in the smooth mapping class group of the boundary Dehn twist on any indefinite symplectic filling of the canonical contact structure on a Seifert-fibered rational homology $3$-sphere. As an application, we deduce the infinite-order non-triviality of the monodromy diffeomorphism of the Milnor fibration of any weighted-homogeneous isolated hypersurface singularity, except for the rational double-point singularities (Theorem \ref{theorem:singularities}). 

Our techniques make essential use of the \textit{Seiberg--Witten--Floer stable homotopy type} introduced by Manolescu \cite{ManolescuStablehomotopytype}, and the $\mathbb{Z}/p$-\textit{equivariant Seiberg--Witten--Floer homology} groups constructed by Baraglia--Hekmati \cite{Baraglia2024equivariant}. A crucial role is played by equivariant versions of the \textit{contact invariant} in the setting of $\mathbb{Z}/p$-equivariant Seiberg--Witten--Floer homology.  

\subsection{The monodromy of Milnor fibrations of weighted hypersurfaces} \label{Intro:singularities}

Let $f : \mathbb{C}^{n+1}\rightarrow \mathbb{C}$ be a polynomial with an \textit{isolated singularity} at $0 \in \mathbb{C}^{n+1}$ and $f(0) = 0$. The vanishing locus $X = V(f)$ is an $n$-dimensional \textit{isolated hypersurface singularity} (abbreviated to IHS).
The local structure near the singular point of $f$ can be probed through the \textit{Milnor fibration} of $f$. By Milnor's celebrated Fibration Theorem \cite{milnor}, the fibers of $f$ assemble into a $C^\infty$ fiber-bundle of (real) $2n$-manifolds with boundary
\begin{align}
f^{-1}(B_{\delta}(\mathbb{C},0)\setminus 0) \cap B_{\epsilon}(\mathbb{C}^{n+1},0) \xrightarrow{f} B_{\delta}(\mathbb{C},0) \setminus 0 \label{milnor1}
\end{align}
with a canonical trivialization of the boundary fibration, for $0 < \delta \ll \epsilon \ll 1$. We refer to $\epsilon , \delta$ as Milnor radii for $f$, which, more precisely, are chosen as follows: $\epsilon$ is chosen so small that for all $0 < \epsilon^\prime \leq \epsilon$ the IHS $V(f)$ intersects transversely with $\partial B_{\epsilon^\prime } (\mathbb{C}^{n+1} , 0 )$, and $\delta$ is chosen so small that for all $t \in B_\delta (\mathbb{C} , 0 )$ the hypersurface $V(f -t )$ is non-singular and intersects $\partial B_\epsilon (\mathbb{C}^{n+1}, 0)$ transversely. Any fiber of (\ref{milnor1}) is called a \textit{Milnor fiber} and denoted by $M$ subsequently. The homotopy type of $M$ is that of a wedge of $\mu$ $n$-spheres, where $\mu$ is the Milnor number of $f$ at $0 \in \mathbb{C}^{n+1}$.

We are interested in  
the {\it monodromy} $\psi \in \diff(M,\partial)$ of the Milnor fibration (\ref{milnor1}), where $\diff(M,\partial)$ denotes the group of diffeomorphisms of $M$ that equal the identity near $\partial M$, and $\psi$ is well-defined up to isotopy fixing $\partial M$. Since the work of Brieskorn \cite{brieskorn,brieskornADE} and Milnor \cite{milnor}, this has been a central topic in the study of singularities by topological methods. However, the classical object of study has only been the induced automorphism $\psi_*$ on the middle-dimensional homology $H_{n}(M;\mathbb{Z})$ (we refer to $\psi_\ast$ as the `classical' monodromy), and considerably less is known about the monodromy diffeomorphism as an element in the smooth mapping class group $\pi_0(\diff(M,\partial))$, especially for $n=2$. This motivates us to probe the monodromy diffeomorphism using techniques from gauge theory. Specifically, the problem we address using gauge theory is to determine whether the monodromy has finite order in $\pi_0(\diff(M,\partial))$.

We will focus on the case of an IHS defined by a \textit{weighted-homogeneous} polynomial, meaning that there exist positive integers $w_1 , \ldots , w_{n+1}, d$ such that 
\[
f(t^{w_1} x_1 , \ldots , t^{w_{n+1}}x_{n+1} ) = t^{d}f(x_1 , \ldots , x_{n+1})
\]
for all $t \in \mathbb{C}$. Weighted-homogeneous polynomials provide a rich class of singularities. Typical examples are the singularities given by the Brieskorn--Pham polynomials $f = \sum_{j = 1}^{n+1} (x_j )^{p_j}$ for $p_j \in \mathbb{Z}_{\geq 2}$, whose {\it links} are the well-known Brieskorn manifolds.
Another famous class is given by the \textit{rational double-point singularities} (the {\it ADE singularities}), defined by the quotient singularities $\mathbb{C}^2/\Gamma$, where $\Gamma$ is a finite subgroup of $ SL(2,\mathbb{C})$ (it is easy to see that these are weighted-homogeneous IHS). The classical monodromy of weighted-homomogeneous IHS singularities has been extensively studied  (see e.g. \cite{brieskorn,milnor,milnor-orlik,brieskornADE,orlik-randell,steenbrink-mixed}). For the rational double-points, the monodromy in $\pi_0 ( \mathrm{Diff}(M, \partial ) )$ is known to have finite order, which follows from Brieskorn's \textit{Simultaneous Resolution} Theorem \cite{brieskornADE}.
A natural question is whether there exist other weighted-homogeneous IHS with finite-order monodromy in the smooth mapping class group.
Our main result shows that the rational double points are, in fact, the {\it only} IHS with this property in complex dimension $2$:

\begin{theorem}\label{theorem:singularities}

Let $X = V(f)$ be a weighted-homogeneous IHS of dimension $n = 2$.
Then the monodromy $\psi$ of the Milnor fibration of $f$ has finite order in the smooth mapping class group $\pi_0 (\mathrm{Diff}(M, \partial ))$ only if $X$ is a rational double-point singularity (i.e. an ADE singularity).
\end{theorem}

Namely, we prove that the monodromy has infinite order in $\pi_0 (\mathrm{Diff}(M, \partial ))$, except for the rational double-points.
It turns out that if the link of the singularity has $b_1 > 0$, then this follows from an elementary argument based on the action on homology.
Thus, the essential part of Theorem~\ref{theorem:singularities} is the case where the link is a rational homology $3$-sphere. In this case, Theorem \ref{theorem:singularities} exhibits a phenomenon which is rather special to the \textit{smooth} category in (real) \textit{dimension} $4$, as the following remarks explain:


\begin{remark}[Topological category in real dimension 4]
If the link is a rational homology $3$-sphere, the monodromy of a weighted-homogeneous IHS has finite order in the \textit{topological} mapping class group $\pi_0 ( \mathrm{Homeo}(M, \partial)  )$.  
This follows from the fact that, 
for a weighted-homogeneous IHS $X$ the action of the monodromy on the middle-degree homology $H_n(M,\mathbb{Z})$ has finite order \cite{milnor}, together with a result of Orson--Powell \cite{orson-powell}. 

\end{remark}

\begin{remark}[Higher dimension]
For an $n$-dimensional weighted-homogeneous IHS whose link is an integral homology $(2n-1)$-sphere with $n \geq 4$, the monodromy $\psi$ has finite order in $\pi_0 (  \mathrm{Diff}(M, \partial ))$ \cite[Proposition 2.22]{KLMME} (for the case $n =4$ see also \cite{krannich2024inftyoperadicfoundationsembeddingcalculus}).
\end{remark}

The rational double-points form a rather special class of weighted-homogeneous singularities. They are characterized by the vanishing of a fundamental analytic invariant of an IHS, namely the \textit{geometric genus} $p_g (X)=0$.
Under the additional hypothesis $p_g(X) = 1$, the infiniteness of the order of the monodromy in $\pi_0(\diff(M,\partial))$ was established in previous work of the authors \cite[Corollary 1.2]{KLMME} using very different techniques that could not handle the case $p_g > 1$.
We proposed an analogous infinite-order non-triviality for $p_g>1$ as \cite[Conjecture 1.3]{KLMME}, which we prove in
Theorem \ref{theorem:singularities}.


Again, the fact that rational double-points have finite order monodromy diffeomorphism follows from the existence of a Simultaneous Resolution for the semi-universal deformation of the rational double-points, after a finite base change \cite{brieskornADE}. In turn, for a IHS with $p_g > 0$ this is well-known to not exist \cite[Proposition 2.17]{KLMME}. Theorem \ref{theorem:singularities} can be regarded as saying that the Simultaneous Resolution property (after base change) doesn't hold when $p_g > 0$ in a differential-topological sense already; thus, providing a \textit{differential-topological explanation} to the non-existence of simultaneous resolutions when $p_g > 0$ under the weighted-homogeneous assumption.
Dropping the latter assumption, we propose the following

\begin{conjecture}
\label{conjecture, not only weighted homogeneous}
If $X = V(f)$ is an IHS of dimension $n =2$ and $X$ is not a rational double point singularity, then the monodromy $\psi$ of the Milnor fibration of $f$ has infinite order in the smooth mapping class group $\pi_0 (\mathrm{Diff}(M, \partial ))$.
\end{conjecture}


Theorem \ref{theorem:singularities} has remarkable implications regarding the \textit{monodromy group} of a $n = 2$-dimensional weighted-homogeneous IHS $X = V(f)$, which we now explain (see \S \ref{section:singularities} for further details). One can `morsify' $f$ by adding in linear terms, i.e. consider $\widetilde{f} = f + \sum_{i =1}^3 a_i x_i$ for a generic choice of coefficients $a_i \in \mathbb{C}$. The subgroup $\Theta (X) \subset \pi_0 ( \mathrm{Diff}(M, \partial) )$ generated by the reflections on the vanishing $2$-spheres of the Morse function $\widetilde{f}$, for some choice of vanishing paths,
does not depend on the choices made. We refer to $\Theta(X)$ as the `geometric' monodromy group of $X$, whereas the `classical' monodromy group $\Theta_h (X)$ is defined as the image of $\Theta(X)$ under the canonical map $h:  \pi_0 ( \mathrm{Diff}(M, \partial ) ) \rightarrow \mathrm{Aut}(H_2 (M, \mathbb{Z} ) )$ with target the automorphism group of the intersection pairing on the middle-degree homology. When $X$ is a rational double-point then the comparison map $\Theta (X ) \rightarrow \Theta_h (X) $ is an isomorphism (and both groups are isomorphic to a suitable Weyl group), which follows from Brieskorn's Theorem. In contrast, Theorem \ref{theorem:singularities} implies the following

\begin{corollary}\label{cor:monodromygroup}
If $X = V(f)$ is a weighted-homogeneous IHS of dimension $n = 2$ whose link is a rational homology 3-sphere and $X$ is not a rational double-point singularity, then the kernel of the comparison map $\Theta (X ) \rightarrow \Theta_h (X)$ contains a $\mathbb{Z}$ subgroup. In particular, the comparison map is not an isomorphism.
\end{corollary}

Theorem \ref{thm: main} can be generalized to equivariant smoothings of weighted-homogeneous isolated surface singularities which are not necessarily hypersurfaces, as well as to IHS's obtained by $\mu$-constant deformations of weighted-homogeneous ones. 
We refer to \S \ref{section:singularities} for these generalization and further discussions.

\subsection{Boundary Dehn twists on indefinite symplectic fillings}\label{Intro:symplectic}
Theorem \ref{theorem:singularities} follows from a more general result about boundary Dehn twists on indefinite symplectic fillings, which we now discuss. 

Let $Y$ be a \textit{Seifert-fibered} rational homology $3$-sphere. This means that $Y$ admits the structure of an orbifold principal $S^1$-bundle $Y \rightarrow C$  over an orbifold closed surface $C$ of genus zero. We orient $Y$ so that the $S^1$-bundle has negative orbifold degree. We refer to this as the \textit{negative orientation}. This orientation on $Y$ is the one arising from any isolated surface singularity with link $Y$. Let $M$ be a compact oriented smooth $4$-manifold with boundary $\partial M = Y$, and consider the associated \textit{boundary Dehn twist} diffeomorphism $\tau_M \in \MCG (M)$ which is described as follows. 
Fix a collar neighbourhood $(0 , 1] \times Y$ of the boundary and a smooth function $\beta : \mathbb{R} \rightarrow \mathbb{R}$ such that $\beta (t) \equiv 0$ near $t = 0$ and $\beta (t) = 2 \pi$ near $t = 1$, and define $\tau_M$ on the collar $(0,1]\times Y$ as
\[
\tau_M (t,y ) = (t , e^{i \beta (t)} \cdot y )
\]
and $\tau_M = \mathrm{Id}$ outside of the collar. This and closely-related diffeomorphism have recently been studied by several authors \cite{KM-dehn,LinK3K3,konno-mallick-taniguchi,KLMME,KangParkTaniguchi,qiu2024dehntwistconnectedsum, miyazawa}. 

We recall that a negatively-oriented Seifert-fibered $3$-manifold $Y$ carries a \textit{canonical contact structure} $\xi_{can}$, which is the unique (up to isotopies) $S^1$-invariant contact structure that is transverse to the fibers of $Y \rightarrow C$. Our main result is the following

\begin{theorem}\label{thm: main}
Let $(Y,\xi )$ be a negatively-oriented Seifert-fibered rational homology 3-sphere equipped with an $S^1$-invariant contact structure $\xi$, and let $(M,\omega )$ be a compact symplectic filling of $(Y,\xi)$ with $b^{+}(M)>0$. Then the boundary Dehn twist $\tau_{M}$ has infinite order in $\pi_0 ( \mathrm{Diff}(M, \partial ) )$.
\end{theorem}

In the situation of Theorem \ref{theorem:singularities}, the link is a negatively-oriented Seifert-fibered rational homology $3$-sphere and the Milnor fiber $M$ provides a symplectic filling of $(Y, \xi_{can} )$ with $b^{+}(M) = 2 p_g (X) > 0$ \cite{durfee}. Theorem \ref{theorem:singularities} now follows from Theorem \ref{thm: main} because by \cite[Proposition 2.14]{KLMME} there exists a non-trivial power of the monodromy $\psi$ which agrees with the boundary Dehn twist $\tau_M$.
Theorem \ref{thm: main} was previously established by the authors \cite{KLMME} under the additional assumptions that $\mathfrak{s}_{\xi} = \mathfrak{s}_{\xi_{can}}$, $\mathfrak{s}_{\xi}$ is self-conjugate and that the reduced monopole Floer homology $HM^{\reduced}_\ast (Y, \mathfrak{s}_{\xi}; \mathbb{Z} )$ is a rank one abelian group. 

By the Theorem of Orson--Powell \cite{orson-powell}, the boundary Dehn twist $\tau_M$ is trivial in the topological mapping class group $\pi_0 ( \mathrm{Homeo}(M, \partial ) )$ when $M$ is simply-connected, and hence Theorem \ref{thm: main} provides examples of infinite order exotic Dehn twist diffeomorphisms.

Finally, about the assumptions in Theorem~\ref{thm: main} we note that, for every $Y$ as in the Theorem, it is easy to construct examples of (i) compact simply-connected symplectic fillings $(M,\omega)$ of $(Y, \xi_{can})$ with $b^{+}(M)=0$ and (ii) compact simply-connected fillings $M$ with $b^{+}(M) > 0$ but not supporting symplectic structures with convex boundary, such that $\tau_M$ is trivial in $\pi_0 ( \mathrm{Diff}(M, \partial ))$. We expect that the $S^1$-invariant hypothesis on $\xi$ can be dropped, and propose the following


\begin{conjecture}
Let $Y$ be a negatively-oriented Seifert-fibered 3-manifold, and let $(M,\omega )$ be a compact symplectic filling of $(Y,\xi)$ with $b^{+}(M)>0$ (where $\xi$ is a contact structure on $Y$). Then the boundary Dehn twist $\tau_{M}$ has infinite order in $\pi_0 ( \mathrm{Diff}(M, \partial ) )$.
\end{conjecture}

\subsection{Outline and Comments}

We outline the key steps in the proof of Theorem \ref{thm: main}, assuming $b_1 (M) =0$ for simplicity. For a negatively-oriented Seifert-fibered rational homology $3$-sphere $Y$, we consider the $G = \mathbb{Z}/p$-action on $Y$ obtained by restriction of the Seifert $S^1$-action to $G \subset S^1$, for a prime number $p$. For the spin-c structure $\mathfrak{s}_Y = \mathfrak{s}_\xi$ and $\mathbb{F} = \mathbb{Z}/p$, we consider the $G$-equivariant Seiberg--Witten--Floer cohomology groups with $\mathbb{F}$-coefficients defined by Baraglia--Hekmati \cite{baraglia2024brieskorn, Baraglia2024equivariant} 
\[ \widecheck{HM}_{G}^\ast (Y, \mathfrak{s}_Y ; \mathbb{F}),\ \widehat{HM}_{G}^\ast (Y, \mathfrak{s}_Y  ; \mathbb{F}) ,\ \overline{HM}_{G}^\ast (Y, \mathfrak{s}_Y ; \mathbb{F})
\]
which are obtained, roughly speaking, as the Borel, coBorel and Tate cohomology, respectively, of the homotopy quotient $SWF(Y, \mathfrak{s}_Y )/\!\!/G := (SWF(Y,\mathfrak{s}_Y) \wedge EG_+)/G$ of the $S^1$-equivariant spectrum $SWF(Y, \mathfrak{s}_Y )$ defined by Manolescu \cite{ManolescuStablehomotopytype}. Thus these are modules over the graded $\mathbb{F}$-algebra $H_{S^1 \times G}^\ast (\ast ; \mathbb{F} ) =  \mathbb{F}[U, R , S]/R^2 $ where $|U| = |S| =2$ and $|R| =1$. Each of these groups can be regarded as a suitable `family Floer homology' of the family of $3$-manifolds $Y \rightarrow (EG\times_{G} Y) \rightarrow BG$. 

Crucially, our proof leverages obstructions coming from a $G$-equivariant analogue of the contact invariant \cite{monolens}. This is an element $c_G (Y, \xi ) \in \widehat{HM}_{G}^\ast (Y, \mathfrak{s}_Y , \mathbb{F})$ canonically attached to a $G$-invariant contact structure $\xi$ on $Y$, and well-defined modulo sign ambiguity. This can be regarded as an invariant associated to the family of contact $3$-manifolds $(Y, \xi ) \rightarrow (EG\times_{G} Y) \rightarrow BG$ which is valued in the family Floer homology. The invariant $c_G (Y,\xi )$ is obtained as the map induced on cohomology by a suitable finite-dimensional approximation of the Seiberg--Witten equation on symplectisations, in the spirit of the homotopy-refinement of the contact invariant studied by Iida and Taniguchi \cite{Iida,IidaTaniguchi} and the families contact invariant studied by the fourth author and Fernández \cite{juan1,juan2}.

Using the fact that $Y/G$ is an $L$-space when $p\gg1$ \cite{baraglia2024brieskorn}, we apply the localization theorem and prove that the image of $c_G (Y, \xi )$ under the natural map $j : \widehat{HM}_{G}^\ast (Y, \mathfrak{s}_Y ) \rightarrow \widecheck{HM}_{G}^\ast (Y, \mathfrak{s}_Y ) $ is $S$-torsion, i.e. $S^N \cdot j ( c_G (Y, \xi ) ) = 0$ for some $N > 0$. On the other hand, assuming for a contradiction that the $m^{th}$ power ($m\neq 0$) of the boundary Dehn twist $\tau_M$ on $M$ is trivial in $\pi_0 ( \mathrm{Diff}(M , \partial  ))$, then for $p \gg 1$ coprime with $m$ there exists a family of $4$-manifolds $M \rightarrow \widetilde{M} \rightarrow BG$ that bounds $Y \rightarrow (EG\times_{G} Y) \rightarrow BG$, by a key result of Kang--Park--Taniguchi \cite{KangParkTaniguchi}. We associate to this a `family cobordism map' in the canonical spin-c structure $\mathfrak{s}_\omega$, which takes the form of a map of $\mathbb{F}[U,R,S]/R^2$-modules 
\[ HM_{G , red}^\ast (Y, \mathfrak{s}_Y ) := \operatorname{Im}(j) \rightarrow \mathbb{F}[U,U^{-1},R,S]/(U,R^2).\] Using the naturality property of the $G$-equivariant contact invariant and Taubes' non-vanishing Theorem for Seiberg--Witten invariants of symplectic manifolds, we show that $j (c_G (Y, \xi ))$ maps to a non-trivial element of $\mathbb{F}[U,U^{-1},R,S]/(U,R^2)$. But the latter module is $S$-torsion-free, which contradicts the fact that $c_G (Y , \xi )$ is $S$-torsion. 

This finishes the proof when $b_{1}(M)=0$. The case $b_{1}(M)>0$ is similar, except that we need to introduce a new construction of a relative Bauer--Furuta invariant for 4-manifold families in the absence of family spin-c structure. This construction may be of independent interest. 


We note that the techniques used in our previous paper \cite{KLMME} are essentially different from the current paper, although they both use family Seiberg--Witten theory. In \cite{KLMME}, we considered the mapping torus of the Dehn twist as a family over $S^1$ and computed its family Seiberg--Witten invariant using the Mrowka--Ozsv\'ath--Yu monopole-divisor correspondence \cite{MOY} and a gluing theorem by the second author \cite{JLIN2022}. That approach only works when $p_{g}=1$ but  has the advantage of being readily generalized to study Dehn twists on closed manifolds. Such generalizations lead to several interesting applications (e.g exotic $\mathbb{R}^4$'s with non-trivial compactly-supported mapping class group). In the present article, we can handle the case $p_{g}>1$ by studying families parametrized by $BG$. However, it remains unclear how to apply this argument to study Dehn twists on closed 4-manifolds.

\subsection*{Acknowledgements:} We are grateful to Javier Fernández de Bobadilla and Chenglong Yu for answering our questions, and David Baraglia, John Etnyre, Nobuo Iida, Jin Miyazawa, and Masaki Taniguchi for helpful comments on a draft of the paper. 
HK was partially supported by JSPS KAKENHI Grant Number 21K13785.
JL was partially supported by NSFC Grant 12271281. 
AM was partially supported by NSF std grant DMS 2405270.


\section{The monodromy groups of weighted singularities}\label{section:singularities}

\subsection{Monodromy Groups of IHS}

We begin by describing in further detail the various notions of monodromy groups discussed in the introduction (for further details we refer to \cite{dimca,arnold}).


Let $X = V(f) \subset \mathbb{C}^{n+1}$ be an $n$-dimensional weighted-homogeneous IHS, and let $\mu \in \mathbb{Z}_{\geq 0}$ denote the Milnor number of $f$ at its isolated singularity $0 \in \mathbb{C}^{n+1}$. Let $\mathcal{X} \xrightarrow{\pi} S$ be its \textit{semi-universal deformation}. Namely, the base $S$ is identified as the vector space underlying the $\mathbb{C}$-algebra
\begin{align}
S  = \frac{\mathbb{C}[x_1, \ldots , x_{n+1}]}{( \partial f / \partial x_1 ,   \ldots , \partial f / \partial x_{n+1} )} \cong \mathbb{C}^\mu  .\label{milnorring}
\end{align}
If we fix a collection $g_1 , \ldots , g_\mu$ of weighted-homogeneous polynomials giving a basis of (\ref{milnorring}), then the total space can be described as the subvariety $\mathcal{X} \subset \mathbb{C}^{n+1}\times \mathbb{C}^\mu$ given by the equations $f(x) + \sum_{i = 1}^{\mu} t_i g_i (x) = 0$ (where $x,t$ denote the standard coordinates on $\mathbb{C}^{n+1}$ and $\mathbb{C}^\mu$), and $\pi$ is the projection to $t$.

From now on, we will work with a suitable representative of the map germ $(\mathcal{X},0 ) \xrightarrow{\pi} (S, 0 )$ (namely, of the form $\pi^{-1}(\mathrm{Int} B_{\delta}(\mathbb{C}^\mu , 0 ) ) \cap \mathrm{Int} B_{\epsilon}(\mathbb{C}^{n+1 +\mu}, 0 ) \xrightarrow{\pi} \mathrm{Int} B_{\delta}(\mathbb{C}^\mu , 0 )$ with $0 < \delta \ll \epsilon \ll 1$). Let $\mathcal{D} \subset S$ denote the \textit{discriminant locus}, which is the reduced irreducible hypersurface in $S$ consisting of those $t \in S$ such that $ \pi^{-1}(t)$ is singular. The $C^\infty$ fiber bundle $\mathcal{X} \setminus \pi^{-1}(\mathcal{D} ) \xrightarrow{\pi} S \setminus \mathcal{D}$ induces a \textit{monodromy homomorphism}
\begin{align}
\rho : \pi_1 ( S \setminus \mathcal{D} ) \rightarrow \pi_0 ( \mathrm{Diff}(M, \partial )  ).
\end{align}
The group $\Theta (X) = \mathrm{Im}(\rho )$ is the \textit{geometric monodromy group} of the IHS $X$. The more classical object of study is the \textit{classical monodromy group} defined as $\Theta_h (X ) = \mathrm{Im}(h \circ \rho )$ where $h$ is the canonical map $h : \pi_0 ( \mathrm{Diff}(M, \partial )  ) \rightarrow \mathrm{Aut}( H_n (M, \mathbb{Z} ) )$. That the definition of $\Theta (X)$ and $ \Theta_h (X)$ given here and the one outlined in the introduction agree follows from \cite[Chapter 3]{arnold}.

The \textit{rational double points can be characterised by the property that $\Theta_h (X)$ is finite}. In this case, Brieskorn's Theorem on simultaneous resolution \cite{brieskorn} asserts that the finite regular covering of $S \setminus \mathcal{D}$ associated to the representation $h \circ \rho$ (with deck group $\Theta_h (X)$) extends to a finite branched covering $\widetilde{S} \xrightarrow{p} S$ (with $\widetilde{S}$ smooth), and that pulling back $\mathcal{X} \rightarrow S$ by $p$ yields a family $\mathcal{X}^\prime \rightarrow \widetilde{S}$ which admits a \textit{simultaneous resolution} (i.e. there exists a \textit{smooth} morphism $\widetilde{\mathcal{X}}\rightarrow \widetilde{S}$ that factors through $\mathcal{X}^\prime \rightarrow \widetilde{S}$ and induces a minimal resolution of singularities of every fiber of the latter map). Furthermore, it follows from this that \textit{when $X$ is a rational double-point then the natural map $\Theta (X) \rightarrow \Theta_h (X)$ is an isomorphism}, which is in sharp contrast with Corollary \ref{cor:monodromygroup}.

\subsection{Equivariant smoothings}

More generally, we can study the monodromy of isolated surface singularities that are not hypersurfaces. Let $(X,0)$ be a weighted-homogeneous isolated singularity (assumed normal). In terms of its ring of regular functions $A$, this means that $A$ is a graded $\mathbb{C}$-algebra $A = \bigoplus_i A_i$, with $A_i = 0$ if $i < 0$ and $A_0 = \mathbb{C}$. The grading specifies an algebraic $\mathbb{C}^\ast$-action with fixed locus $0$, which is `good' in the sense that $0$ is repelling fixed point. A weighted-homogeneous isolated singularity can be described as an affine quasicone on a quasismooth subvariety in a weighted projective space \cite{dolgachev82,pinkham}, and the singular point is the vertex.

Theorem \ref{theorem:singularities} can be generalised to the case of \textit{equivariant smoothings} of weighted-homogeneous isolated singularities. We recall (see also \cite[\S 2.2.3]{KLMME} for a slightly more general definition) that a smoothing $ f: (\mathcal{X},0) \rightarrow (\mathbb{C},0)$ is equivariant if there exists an algebraic $\mathbb{C}^\ast$-action on $\mathcal{X}$ extending the given one on $X$, and an algebraic $\mathbb{C}^\ast$-action on $\mathbb{C}$ fixing $0$, such that $f$ is equivariant. By \cite[Proposition 2.14]{KLMME}, a non-trivial power of the monodromy of such a smoothing agrees in $\pi_0 ( \mathrm{Diff}(M, \partial ))$ with the boundary Dehn twist on the Milnor fiber. Thus, Theorem \ref{thm: main} implies the following result:

\begin{theorem}\label{thm:singularities2}
Let $X$ be a weighted-homogeneous $n = 2$-dimensional isolated singularity which is not rational (i.e. $p_g (X, 0) > 0$). Then the monodromy diffeomorphism $\psi$ of any equivariant smoothing of $X$ has infinite order in $\pi_0 ( \mathrm{Diff}(M, \partial ) )$. 
\end{theorem}

We recall, for context, that whether an isolated $2$-dimensional singularity is rational (i.e. $p_g (X_0 , 0 ) = 0$) only depends on its link $Y$. When $Y$ is a rational homology $3$-sphere, this occurs precisely when $Y$ is an $L$-space \cite{OzsvathSzaboPlumbed,nemethi-rational}.

\begin{example}
Consider a $2$-dimensional \textit{isolated complete intersection singularity} (abbreviated to ICIS), namely $X \subset \mathbb{C}^{2+k}$ is an $2$-dimensional subvariety with an isolated singularity at $0 \in \mathbb{C}^{2+k}$ given as $X = V(f_1 , \ldots , f_k )$ for a collection of polynomials $f = (f_1 , \ldots , f_k): \mathbb{C}^{2+k}\rightarrow \mathbb{C}^k $ with $f(0) = 0$. The ICIS is weighted-homogeneous if each $f_1 , \ldots f_k$ can be taken to be weighted-homogeneous with the same $w_1 , \ldots , w_{2 + k}$ for all $f_i$, and possibly different $d_i$ for each $f_i$. In this situation, there always exists an equivariant smoothing of $X$. Namely, consider $\mathcal{X} \subset \mathbb{C}^{2 + k}\times \mathbb{C}_t$ given by taking the equations $f_i (x) = c_i t^{d_i}$, for generic coefficients $c_i \in \mathbb{C}$, where $t$ is the smoothing parameter.

For a simple example, fix pairwise coprime $p_j \in \mathbb{Z}_{\geq 2}$, $j = 1 , \ldots , 2+k$, and consider the $2$-dimensional weighted-homogeneous ICIS cut out by equations $f_i = \sum_{j = 1}^{k} a_{ij} x_{j}^{p_j} $, $j =1 , \ldots , k$, for a generic choice of coefficients $a_{ij}\in \mathbb{C}$. Its link is the Seifert-fibered integral homology $3$-sphere denoted $\Sigma(p_1 , \ldots, p_{2+k})$, which is not an $L$-space except for the case $k = 1$ and $(p_1 , p_2 , p_3 ) = (2,3,5)$ (i.e. the Poincaré sphere $\Sigma(2,3,5)$).

\end{example}

\subsection{$\mu$-constant deformations}

Using Theorem \ref{theorem:singularities}, we can deduce infinite-order results for the mononodromy diffeomorphisms of IHS which are given by $\mu$-constant deformations of a weighted-homogeneous IHS. We discuss such generalizations now. We first recall this notion of deformation. We consider a complex polynomial $f : \mathbb{C}^{n+1}\rightarrow \mathbb{C}$ with an isolated critical point at $0 \in \mathbb{C}^{n+1}$ with $f(0) = 0$.

\begin{definition}
A $\mu$-\textit{constant deformation} of $f$ consists of a family of complex polynomials $f_s : \mathbb{C}^{n+1}\rightarrow \mathbb{C}$ with $f = f_0$ such that the coefficients of $f_s$ depend smoothly on $s \in [0,1]$, each $f_s$ has an isolated critical point at $0 \in \mathbb{C}^{n+1}$ with $f_s (0) = 0$, and the Milnor number $\mu (f_s , 0 )$ of $f_s$ at $0 \in \mathbb{C}^{n+1}$ is independent of $s \in [0,1]$. Here, the Milnor number of $f$ at $0\in \mathbb{C}^{n+1}$ is defined as the dimension of the complex vector space $(\ref{milnorring})$.
\end{definition}

Let $M_s$ denote the Milnor fiber of $f_s$. The following result first appeared in an influential work by Le--Ramanujam \cite{Le-Ramanujam}:

\begin{lemma}\label{lemma:le-ramanujam}
Let $f_s$ be a $\mu$-constant deformation of $f = f_0$. Then, for sufficiently small $s >0$, there exists a smooth and proper embedding $M_s \hookrightarrow M_0$ such that the monodromy diffeomorphism $\psi_0 \in \pi_0 ( \mathrm{Diff}(M_0 , \partial ) )$ of the Milnor fibration of $f_0$ agrees with the monodromy diffeomorphism $\psi_s  \in \pi_0 ( \mathrm{Diff}(M_s , \partial ) )$ of the Milnor fibration of $f_s$, after extending the latter by the identity to a diffeomorphism of $M_0$.

\end{lemma}

\begin{proof}
This follows from the Le--Ramanujam `Non-Splitting' Theorem \cite{Le-Ramanujam}. This asserts that there exist Milnor radii $\epsilon_0 , \delta_0$ for $f = f_0$ and $s_0 > 0$ so small that for $F (x,s):= f_s (x)$, the mapping
\begin{align*}
B_{\epsilon_0}(\mathbb{C}^{n+1}, 0 ) \times [0,s_0 ] \cap F^{-1}\Big( ( B_{\delta_0}(\mathbb{C}, 0 ) \setminus 0 ) \times [0, s_0 ] \Big) \xrightarrow{F} ( B_{\delta_0}(\mathbb{C}, 0 ) \setminus 0 ) \times [0, s_0 ] 
\end{align*}
is a smooth fiber bundle of $2n$-manifolds with boundary, and with a canonical trivialisation of the boundary fibration as there is also a smooth fiber bundle
\begin{align*}
S_{\epsilon_0}(\mathbb{C}^{n+1}, 0 ) \times [0,s_0 ] \cap F^{-1}\Big( B_{\delta_0}(\mathbb{C}, 0 )  \times [0, s_0 ] \Big) \xrightarrow{F}  B_{\delta_0}(\mathbb{C}, 0 )  \times [0, s_0 ] .
 \end{align*}
 and the base of the latter is contractible. Thus, for any $s \in [0, s_0 ]$ the fiber bundle defined by 
\begin{align}
B_{\epsilon_0}(\mathbb{C}^{n+1}, 0 )  \cap f_{s}^{-1} ( B_{\delta_0}(\mathbb{C}, 0 ) \setminus 0 ) \xrightarrow{f_s}  B_{\delta_0}(\mathbb{C}, 0 ) \setminus 0 \label{milnors} 
\end{align}
is equivalent to the Milnor fibration (\ref{milnor1}) of $f = f_0$. For a fixed $s \in (0 , s_0 ]$ and a choice of Milnor radii $\epsilon_s, \delta_s$ for $f_s$, the Milnor fibration of $f_s$ is 
\begin{align}
B_{\epsilon_s}(\mathbb{C}^{n+1}, 0 )  \cap f_{s}^{-1} ( B_{\delta_s}(\mathbb{C}, 0 ) \setminus 0 ) \xrightarrow{f_s}  B_{\delta_s}(\mathbb{C}, 0 ) \setminus 0  \label{milnorfs}.
\end{align}
After possibly shrinking the radii $\epsilon_s , \delta_s$ further, we may suppose $ \epsilon_s < \epsilon_0$ and $\delta_s \leq \delta_0$. Thus, we obtain an embedding of (\ref{milnorfs}) into (\ref{milnors}). This completes the proof of the first assertion. The last assertion follows from the fact that 
\begin{align}
( B_{\epsilon_0}(\mathbb{C}^{n+1}, 0 ) \setminus \mathrm{Int} B_{\epsilon_s}(\mathbb{C}^{n+1}, 0 ) )\cap f_{s}^{-1} ( B_{\delta_s}(\mathbb{C}, 0 ) ) \rightarrow  B_{\delta_s}(\mathbb{C}, 0 ) 
\end{align}
is a smooth fiber bundle over a contractible base.
\end{proof}

\begin{remark}
The cobordism $W := M_0 \setminus M_s$ is naturally a Stein cobordism from $\partial M_s$ to $\partial M_0$ with respect to the canonical contact structures on the singularity links. By the main result of \cite{Le-Ramanujam}, when $n \neq 2$ the cobordism $W$ is diffeomorphic to a product. Indeed, when $n > 2$ they show that $W$ is an $h$-cobordism between simply-connected $(2n-1)$-dimensional manifolds, and thus the h-cobordism Thereom applies. When $n = 2$ the cobordism $W = M_0 \setminus M_s$ is an integral homology cobordism and it is an open problem whether, in general, $W$ is homeomorphic to a product cobordism (this is the famous Le--Ramanujam problem).
\end{remark}

Combining Lemma \ref{lemma:le-ramanujam} with Theorem \ref{theorem:singularities} we immediately deduce the following generalization:

\begin{theorem}
Suppose that $X = V(f)$ is a weighted-homogeneous IHS of dimension $n = 2$ and $X$ is not a rational double-point. If $f_s$ is a $\mu$-constant deformation of $f = f_0$, then for sufficiently small $s > 0$ the monodromy diffeomorphism of the Milnor fibration of $f_s$ has infinite order (in the smooth mapping class group of the Milnor fiber of $f_s$).
\end{theorem}

The $\mu$-constant deformations of weighted-homogeneous IHS are well-understood, as we now recall. Let $w_i , d \in \mathbb{Z}_{>0}$ stand for the weights of $f$. There is an induced grading on the base $S$ of the semi-universal deformation of the weighted-homogeneous IHS $V(f)$ (described in (\ref{milnorring}) ) induced by $|x_i | = w_i $. As before, we fix a basis of $S$ given by a collection $g_1 , \ldots , g_\mu$ of weighted-homogeneous polynomials. Take a smooth family $f_s$ of complex polynomials with $f = f_0$ parametrised by $s \in [0,1]$ which is given by a smooth map $t: [0,1] \rightarrow S$, i.e. $f_s (x) = f(x) + \sum_{i = 1}^\mu t_i (s) g_i (x)$. By a Theorem of Varchenko \cite{varchenko}, $f_s$ is a $\mu$-constant deformation (for small $s \geq 0$) precisely when the weight $|g_i|$ of $g_i$ is at least $d$ whenever $t_i (s)$ is not identically zero near $ s = 0$.

\begin{example}
The base $S \cong \mathbb{C}^{12}$ of the semi-universal deformation of $V(x^2 + y^3 + z^7 )$ has exactly one weight $\geq d$, namely $|yz^5| = w_2 + 5w_3 = 14 + 5 \times 6 = 44 \geq 42 = d$. Thus, $f_s = x^2 + y^3 + z^7 + s y z^5$ gives a $\mu$-constant deformation (for small $s \geq 0$) of $f_0 = x^2+ y^3 +z^7$.
\end{example}

\section{Equivariant monopole Floer homology of Seifert manifolds}

We start by discussing the equivariant monopole Floer homology of Seifert manifolds. Let $Y$ be a negatively oriented Seifert rational homology 3-sphere and let $\fs$ be a spin-c structure $Y$. We fix an odd prime $p$ and let $G=\F=\mathbb{Z}/p$. Consider the Seifert $G$-action on $Y$ (i.e. the $G$-action restricted from the Seifert $S^1$-action). We choose $p$ such that $G$ does not intersect the isotropy groups of the Seifert $S^1$-action. In this case, the $G$-action is free on $Y$ and the quotient space $Y/G$ is another Seifert manifold. It is shown in \cite{baraglia2024brieskorn} that the $G$-action can be lifted to the spinor bundle over $Y$. We fix such a lifting and let $\mathfrak{s}/G$ be the spin-c structure on $Y/G$ obtained by taking the quotient with this lifting. 

Let us briefly recall the Seiberg--Witten-Floer spectrum $\SWF(Y,\fs)$ of $(Y,\mathfrak{s})$ due to Manolescu \cite{ManolescuStablehomotopytype}. 
This spectrum $\SWF(Y,\fs)$ is $S^1$-equivariant, allowing suspensions by copies of the real 1-dimensional trivial $S^1$-representation $\mathbb{R}$ and the standard complex 1-dimensional $S^1$-representation $\mathbb{C}$.
The main ingredient of this construction is the Conley index $I_{\lambda}^\mu(Y,\fs,g)$, which is a pointed finite CW complex with an $S^1$-action, of a dynamical system of a finite-dimensioal approximation of an infinite-dimensional dynamical system associated to the Chern-Simons-Dirac functional on $Y$. 
The Conley index $I_\lambda^\mu(Y,\fs,g)$ depends on the choice of metric $g$ and real numbers $\mu, -\lambda>0$, which correspond to the choice of finite-dimensional approximation.  A metric-dependent spectrum $\SWF(Y,\fs,g)$ is defined as $\Sigma^{-V^0_{\lambda}(Y,\fs,g)}I_{\lambda}^{\mu}(Y,\fs,g)$, where $\Sigma^{-V^0_{\lambda}(Y,\fs,g)}$ denotes the formal desuspension by the vector space $V^0_{\lambda}(Y,\fs,g)$, the span of eigenspaces of the linearization of the Seiberg--Witten equations on $Y$ with eigenvalues in the interval $(\lambda, 0]$. Using $\SWF(Y,\fs,g)$, a metric-independent spectrum $\SWF(Y,\fs)$ is defined as
\[
\SWF(Y,\fs) = \Sigma^{n(Y,\fs,g)\mathbb{C}}\SWF(Y,\fs,g),
\]
where $\Sigma^{n(Y,\fs,g)\mathbb{C}}$ denotes the formal suspension by $n(Y,\fs,g)$-copies of $\mathbb{C}$ and $n(Y,\fs,g) \in \mathbb{Q}$ is an index-theoretic correction term:
\begin{align}
\label{eq: n(Y,g)}
n(Y,\fs,g)=\operatorname{ind}_{\mathbb{C}}(D^+_M)-\frac{c^2_{1}(\fs_{M})-\sigma(M)}{8}\in \mathbb{Q}.    
\end{align}
Here $(M,\fs_{M})$ is a smooth spin-c 4-manifold bounded by $(Y,\fs)$, $D^+_M$ is the spin-c Dirac operator of $(M,\fs_{M})$, and $\operatorname{ind}_{\mathbb{C}}(D^+_M)$ is the Atiyah--Patodi--Singer index. By the theorem of Lidman--Manolescu \cite{LidmanManolescuEquivalence}, various versions of $S^{1}$-equivariant homology of $\SWF(Y,\fs)$ recover various flavours of monopole Floer homology of $(Y,\fs)$ \cite{KronheimerMonopoles}, for example
\begin{align}
\label{eq: from SWF to HM}
\widecheck{HM}^{*}(Y,\fs) \cong \tilde{H}_{S^1}^\ast(\SWF(Y,\fs)) := \tilde{H}_{S^1}^{\ast-2n(Y,\fs,g)+\dim V^0_{\lambda}(Y,\fs,g)}(I_{\lambda}^{\mu}(Y,\fs,g)).
\end{align}

Henceforth, we freely use the standard notation of monopole Floer homology in \cite{KronheimerMonopoles}.

Next, we consider the $G$-equivariant monopole Floer homology.
From now we assume the metric $g$ on $Y$ is taken to be $G$-invariant.

Given a pointed $G$-space $Z$, let $Z /\!\!/ G$ denote the homotopy quotient, $Z /\!\!/ G =(Z \wedge EG_+)/G$.
As proved by Baraglia--Hekmati \cite{Baraglia2024equivariant}, by doing $G$-equivariant finite-dimensional approximations of Seiberg--Witten equations and taking $G$-invariant Conley index, one can make $I^{\mu}_{\lambda}(Y,\mathfrak{s},g)$ into a finite $(G\times S^1)$-CW complex. Then $I^{\mu}_{\lambda}(Y,\mathfrak{s},g)/\!\!/ G$ is an infinite $S^{1}$-CW complex. 

One can define various equivariant Floer 
cohomology groups of the spin-c $3$-manifold $(Y,\fs)$ with $G$-action, by applying to $I^{\mu}_{\lambda}(Y,\mathfrak{s},g)/\!\!/ G$ the reduced $S^1$-Borel cohomology $\tilde{H}^{\ast}_{S^{1}}(-)$, the reduced $S^1$-coBorel cohomology $c\tilde{H}^{\ast}_{S^{1}}(-)$, the reduced Tate cohomology $t\tilde{H}^{\ast}_{S^{1}}(-)$, and the reduced nonequivariant cohomology $\tilde{H}^{\ast}(-)$. More precisely, we define 
\begin{equation}\label{eq: equivariant HM}
\begin{split}
\widecheck{HM}^{\ast}_{G}(Y,\fs;\F):=&\tilde{H}^{\ast-2n(Y,\fs,g)+\dim V^0_{\lambda}(Y,\fs,g)}_{S^{1}}(I^{\mu}_{\lambda}(Y,\mathfrak{s},g)/\!\!/ G;\F), \\
\widehat{HM}^{\ast}_{G}(Y,\fs;\F):=&c\tilde{H}^{\ast-2n(Y,\fs,g)+\dim V^0_{\lambda}(Y,\fs,g)}_{S^{1}}(I^{\mu}_{\lambda}(Y,\mathfrak{s},g)/\!\!/ G;\F),\\
\overline{HM}^{\ast}_{G}(Y,\fs;\F):=&t\tilde{H}^{\ast-2n(Y,\fs,g)+\dim V^0_{\lambda}(Y,\fs,g)}_{S^{1}}(I^{\mu}_{\lambda}(Y,\mathfrak{s},g)/\!\!/ G;\F),\\
\widetilde{HM}^{*}_{G}(Y,\fs;\F):=&\tilde{H}^{\ast-2n(Y,\fs,g)+\dim V^0_{\lambda}(Y,\fs,g)}(I^{\mu}_{\lambda}(Y,\fs,g)/\!\!/ G ;\F).
\end{split}
\end{equation}
As one changes the auxiliary choices $g,\mu,\lambda$, the space $I^{\mu}_{\lambda}(Y,\fs,g)$ changes by $G$-equivariant homotopy equivalences and suspension/desuspension by finite dimension $G$-representations. Since the order of $G$ is odd, the $G$-actions preserve the orientations of these representation spaces. By the Thom isomorphism theorem, the grading normalization in (\ref{eq: equivariant HM}) makes these equivariant Floer cohomology invariant of $(Y,\fs)$.

Note the fibration 
\begin{equation}\label{eq: SWF fibration}
I_{\lambda}^{\mu}(Y,\fs,g) 
\xrightarrow{q} (I_{\lambda}^{\mu}(Y,\fs,g)\times EG)/G
\xrightarrow{\pi}
BG.    
\end{equation}
This makes $\widetilde{HM}^{*}_{G}(Y,\fs;\F)$ into a graded module over the ring 
\begin{equation}\label{eq: Q1}
Q_{1}:=H^{\ast}_{G}(\text{point};\F)=H^{*}(BG;\F)=\F[R,S]/(R^2)    
\end{equation}
and makes $\widecheck{HM}^{\ast}_{G}(Y,\fs;\F)$, $\widehat{HM}^{\ast}_{G}(Y,\fs;\F)$ and $\overline{HM}^{\ast}_{G}(Y,\fs;\F)$ into graded modules over the ring 
\begin{equation}\label{eq: Q2}
Q_2:=H^{\ast}_{S^{1}\times G}(\text{point};\F)=H^{\ast}_{S^{1}}(BG;\F)=\F[U,R,S]/(R^2).   \end{equation}
Here $\operatorname{deg}(S)=\operatorname{deg}(U)=2$ and $\operatorname{deg}(R)=1$. 

By standard property of Borel, coBorel and Tate cohomology \cite{GreenlessMay}, we have the exact triangle of $Q_{2}$-modules 
\begin{equation}\label{eq: exact triangle 1}
\xymatrix{&\widecheck{HM}^{\ast}_{G}(Y,\fs;\F) \ar[rd]^{i}&\\
\widehat{HM}^{\ast}_{G}(Y,\fs;\F)\ar[ur]^{j}& & \overline{HM}^{\ast}_{G}(Y,\fs;\F),\ar[ll]^{k}
}
\end{equation}
with $\operatorname{deg}(k)=1$, $\operatorname{deg}(i)=\operatorname{deg}(j)=0$. And we have  exact triangles of $Q_{1}$-modulies
\begin{equation}\label{eq: exact triangle 2}
\xymatrix{&\widetilde{HM}^{\ast}_{G}(Y,\fs;\F) \ar[rd]^{l}&\\
\widecheck{HM}^{\ast}_{G}(Y,\fs;\F)\ar[ur]^{m}& & \widecheck{HM}^{\ast}_{G}(Y,\fs;\F)\ar[ll]^{\cdot U}}
\end{equation}
and 
\begin{equation}\label{eq: exact triangle 3}
\xymatrix{&\widetilde{HM}^{\ast}_{G}(Y,\fs;\F) \ar[rd]^{p}&\\
\widehat{HM}^{\ast}_{G}(Y,\fs;\F)\ar[ur]^{n}& & \widehat{HM}^{\ast}_{G}(Y,\fs;\F),\ar[ll]^{\cdot U}}
\end{equation}
with $\operatorname{deg}(l)=\operatorname{deg}(p)=-1$ and $\operatorname{deg}(m)=\operatorname{deg}(n)=0$.

The reduced $G$-equivariant monopole Floer homology is defined as  \[HM^{*}_{G,\redu}(Y,\fs;\F):=\operatorname{Im} (k:\widehat{HM}^{\ast}_{G}(Y,\fs;\F)\to \widecheck{HM}^{\ast}_{G}(Y,\fs;\F)).\]
It can be regarded either as a quotient $Q_2$-module of  $\widehat{HM}^{\ast}_{G}(Y,\fs;\F)$ or a sub $Q_2$-module of $\widecheck{HM}^{\ast}_{G}(Y,\fs;\F)$. 

Note that the fiber inclusion map $q$ (\ref{eq: SWF fibration}) induces natural transformations from various $G$-equivariant monopole Floer cohomologies to various nonequivariant monopole Floer cohomologies. In particular, we have a map of $\F[U]$-modules  
\begin{equation}\label{eq: equivariant red to red}
q^*:HM^{*}_{G,\redu}(Y,\fs;\F)\to HM^{*}(Y,\fs;\F)    
\end{equation}

\begin{lemma}\label{lem: S-localization}
For $p$ large enough, we have 
\[
S^{-1}\widetilde{HM}^{*}_{G}(Y,\fs;\F)\cong \F[S,S^{-1},R]/(R^2)
\]
as relative $\mathbb{Z}$-graded $Q_1$-modules.
\end{lemma}
\begin{proof}
The localization theorem in $G$-equivariant cohomology \cite[III, Theorem 3.8]{tomDieck}, applied to the finite $G$-CW complex $I^{\mu}_{\lambda}(Y,\fs,g)$, implies the isomorphism
\[
S^{-1}\widetilde{HM}^{*}_{G}(Y,\fs;\F)=S^{-1}\tilde{H}^{*}_{G}(I^{\mu}_{\lambda}(Y,\fs,g);\F)\cong S^{-1}\tilde{H}^{*}_{G}(I^{\mu}_{\lambda}(Y,\fs,g)^{G};\F)
\]
of relative $\mathbb{Z}$-graded $Q_{1}$-modules. Here $I^{\mu}_{\lambda}(Y,\fs,g)^{G}$ denotes the space of $G$-fixed points. It is proved in \cite{LidmanManolescu} that 
\[I^{\mu}_{\lambda}(Y,\fs,g)^{G}\simeq I^{\mu}_{\lambda}(Y/G,\fs/G,g/G).\] 
So we have the isomorphism 
\[S^{-1}\widetilde{HM}^{*}_{G}(Y,\fs;\mathbb{F})\cong \widetilde{HM}^{*}(Y/G,\fs/G;\mathbb{F})\otimes_{\F} \F[S,S^{-1},R]/(R^2)\]
as relative $\mathbb{Z}$-graded modules. On the other hand, we have 
\[
HM_{\reduced}(Y/G,\fs_0;\mathbb{F})=0
\]
when $p\gg 0$ \cite[Proposition 6.2]{baraglia2024brieskorn}. Here $\fs_0$ denotes the canonical spin-c structure on $Y/G$. So $Y/G$ must be an $L$-space because otherwise $\fs_0$ would support a taut foliation \cite{lisca-stipsicz}, and thus we would have $HM_{\redu}(Y/G , \mathfrak{s}_0 ; \mathbb{F} ) \neq 0$ by \cite{monolens}. So we have $\widetilde{HM}^{*}(Y/G,\fs/G;\mathbb{F})\cong \F$ and 
\[
S^{-1}\widetilde{HM}^{*}_{G}(Y,\fs;\mathbb{F})\cong \F[S,S^{-1},R]/(R^2).
\] \end{proof}
From now on, we always assume $p$ is large enough such that Lemma \ref{lem: S-localization} can be applied. 
\begin{lemma}\label{lem: Sm surjective}
The map 
\[
S^{-1}m:
S^{-1}\widecheck{HM}^{*}_{G}(Y,\fs;\mathbb{F})\to S^{-1}\widetilde{HM}^{*}_{G}(Y,\fs;\mathbb{F}),
\]
obtained by inverting the map $m$
in (\ref{eq: exact triangle 2}), is surjective.
\end{lemma}
\begin{proof}
A spectral sequence from the nonequivariant monopole Floer cohomology to the equivariant monopole Floer cohomology is established in \cite{Baraglia2024equivariant}, which arises as the Serre spectral sequence for (\ref{eq: SWF fibration}). We use integer coefficients for this spectral sequence. Then each page is a module over $Q=\mathbb{Z}[U,S]/(pS)$, with $U$ of grading $(0,2)$ and $S$ of grading $(2,0)$. Let $h(Y,\fs)$ be the Fr\o yshov invariant of $(Y,\fs)$. We normalize the grading such that
\[
E^{r,q}_{2}=H^{r}(BG;\widecheck{HM}^{-2h(Y,\fs)+q}(Y,\fs;\mathbb{Z})),\quad r\in \mathbb{N}, q\in \mathbb{Z},
\]
Here $h(Y,\fs)$ denotes the Fr\o yshov invariant. The spectral sequence converges to $\widecheck{HM}_{G}^{\ast} (Y,\fs;\mathbb{Z})$. Since the generator of the $G$-action on $Y$ is isotopic to the identity, the monodromy action of (\ref{eq: SWF fibration}) on $\widecheck{HM}^{*}(Y,\fs;\mathbb{Z})$ is trivial. Moreover, $\widecheck{HM}^{*}(Y,\fs;\mathbb{Z})$ can be computed by the graded roots algorithm \cite{Nemethi11}. In particular, the two summands of the decomposition
\[
\widecheck{HM}^{*}(Y,\fs;\mathbb{Z})\cong \mathbb{Z}[U]_{-2h(Y,\fs)}\oplus HM^{*}_{\reduced}(Y,\fs;\mathbb{Z}),
\]
are supported in different relative $\mathbb{Z}/2$-grading gradings. This decomposition induces a decomposition $E^{*,*}_{2}=E^{*,\text{even}}_{2}\oplus E^{*,\text{odd}}_{2}$, where 
\[
E^{*,\text{odd}}_{2}=(HM^{*}_{\reduced}(Y,\fs;\mathbb{Z})\otimes_{\mathbb{Z}}H^{*}(BG;\mathbb{Z})))\cong HM^{*}_{\reduced}(Y,\fs;\mathbb{Z})[S]/(pS),
\]   
and
\[
E^{*,\text{even}}_{2}=(\mathbb{Z}[U]\otimes_{\mathbb{Z}} H^{*}(BG;\mathbb{Z}))\cong Q\cdot \theta.
\]
Here $\theta$ is a generator of $\mathbb{Z}\cong E^{0,0}_{2}$.  Note that $E^{r,*}_{2}=0$ if $r$ is odd. So $E^{*,\text{even}}_{2}, E^{*,\text{odd}}_{2}$ are supported in different mod-2 total degrees. Furthermore, since $E^{*,\text{odd}}_{2}$ is $U$-torsion but $E^{*,\text{even}}_{2}$ is $U$-torsion-free. Thus, the differential on $E^{*,*}_{2}$ is of the form 
\[
d_{2}: E^{*,\text{even}}_{2}\to E^{*,\text{odd}}_{2}.
\]
Hence $E^{*,\text{even}}_{3}=\ker d_{2}$  is $U$-torsion-free and $E^{*,\text{odd}}_{3}=\operatorname{coker}d_{2}$ is $U$-torsion. Similarly, the differential on $E^{*,*}_{k}$ has the form 
\[
d_{k}: E^{*,\text{even}}_{k}\to E^{*,\text{odd}}_{k}.
\]
We will inductively prove that the following claim: 
For each even $k\geq 2$, there exists an integer $q_{k}\geq 0$ and a quotient $\mathbb{F}[U]$-module $V_{k}$ of $HM^{*}_{\reduced}(Y,\fs;\F)$ such that the following conditions hold
\begin{itemize}
        \item For any $r>0$ and any even $q<2q_{k}$, $E^{r,q}_{k}=0$.
        \item $\bigoplus_{q \text{ even }, q\geq 2q_{k}, r\geq 0}E_{k}^{r,q}$ is a free $Q$-module generated by $U^{q_{k}}\theta\in E^{0,2q_{k}}_{2k}$.
         \item $\bigoplus_{q \text{ odd }, r\geq k}E_{k}^{r,q}= \F[S]S^{k/2}\cdot V_k$.
\end{itemize}
When $k=2$, the claim is easily verified by setting $q_{2}=0$ and $V_{2}=HM^{*}_{\reduced}(Y,\fs;\F)$. Suppose we have proved the case $k=n$. The differential 
\begin{equation}
d_{k}: \bigoplus_{q \text{ even }, q\geq 2q_{k}, r\geq 0}E^{r,q}_{k}\to \bigoplus_{q \text{ even }, q\geq 2q_{k}-k+1, r\geq k}E_{k}^{r,q}    
\end{equation}
sends $U^{q_{k}}\theta\in E^{0,2q_{k}}_{k}$ to $S^{k/2}\cdot a_{k}\in E^{k,2q_{k}-k+1}_{k}$ for some $a_{k}\in V_{k}$. We let 
\[
n_{k}=\min\{q\geq 0\mid U^{q}a_{k}=0\in V_{k}\}
\]
Then one can verify that $q_{k+2}=q_{k}+n_{k}$ and $V_{k+2}=V_{k}/(a_{k})$ satisfies the requirement. This finishes the induction and the claim is proved.

Since $HM^{*}_{\reduced}(Y,\fs;\F)$ is a finite dimensional $\mathbb{F}$-vector space, then the spectral sequence must degenerate after finitely many pages. It follows from the claim that for some number $q_{\infty}$, \[\bigoplus_{q \text{ even},\ q\geq 2q_{\infty},\ r\geq 0}E_{\infty}^{r,q}\] is a free $Q$-module generated by $U^{q_{\infty}}\cdot\theta$, and 
\[ 
\bigoplus_{q \text{ even}, \ q < 2 q_\infty , \ r > 0} E_{\infty}^{r,q}= 0 .
\]

Therefore, 
\[
\widecheck{HM}^{-2h(Y,s)+2q_{\infty}+2\mathbb{N}}_{G}(Y,\fs;\mathbb{Z}):=\bigoplus_{q\in -2h(Y,s)+2q_{\infty}+2\mathbb{N}}\widecheck{HM}^{q}_{G}(Y,\fs;\mathbb{Z})
\]
is a free $Q$-module of rank $1$. By the $\mathbb{Z}$-version of (\ref{eq: exact triangle 2}), we see that image of the map 
\[
m_{\mathbb{Z}}:\widecheck{HM}^{*}_{G}(Y,\fs;\mathbb{Z})\to \widetilde{HM}^{*}_{G}(Y,\fs;\mathbb{Z})
\]
contains a $\mathbb{Z}[S]/(pS)$-module isomorphic to $\F[S]$, supported in degrees $-2h(Y,\s)+2 q_{\infty}+2\mathbb{N}$. Inverting $S$ and applying the change of coefficient map, we see that the map 
\[
S^{-1}m: S^{-1}\widecheck{HM}^{*}_{G}(Y,\fs;\mathbb{Z})\to S^{-1}\widetilde{HM}^{*}_{G}(Y,\fs;\mathbb{F})\cong \F[S,S^{-1},R]/(R^2)
\]
is surjective.
\end{proof}

\begin{corollary}\label{cor: kerU annilated by S}
For any element $a$ in the kernel of the map \[U: HM^{*}_{G,\redu}(Y,\fs;\F)\to HM^{*}_{G,\redu}(Y,\fs;\F)),\] there exists  $N$ such that $S^{N}a=0$.    
\end{corollary}
\begin{proof} We treat $HM^{*}_{G,\redu}(Y,\fs;\F)$ as a sub $Q_{2}$-module of $\widecheck{HM}^{\ast}_{G}(Y,\fs;\F)$. We invert $S$ in (\ref{eq: exact triangle 2}) and obtain the exact triangle 
\begin{equation*}
\xymatrix{&S^{-1}\widetilde{HM}^{\ast}_{G}(Y,\fs;\F) \ar[rd]^{S^{-1}l}&\\
S^{-1}\widecheck{HM}^{\ast}_{G}(Y,\fs;\F)\ar[ur]^{S^{-1}m}& & S^{-1}\widecheck{HM}^{\ast}_{G}(Y,\fs;\F)\ar[ll]^{\cdot U}}
\end{equation*}
By exactness, $a$ belongs to the image of $S^{-1}l$.
By Lemma \ref{lem: Sm surjective}, the map $S^{-1}l$ is trivial. So $a=0\in S^{-1}\widecheck{HM}^{\ast}_{G}(Y,\fs;\F)$.
\end{proof}


\begin{example}
It might be instructive to describe the spectral sequence from the proof of Lemma \ref{lem: Sm surjective} for the Brieskorn homology sphere $Y = \Sigma(2,3,7)$. This has $h(Y) = 0$ and $HM_{red}^\ast (Y) $ given by $\mathbb{Z}$ in degree $-1$ and zero otherwise. As bigraded $Q$-modules, we thus have 
\[
E_{2}^{\ast , even } = Q\langle \theta \rangle \quad , \quad E_{2}^{\ast , odd} = Q \langle x \rangle / (U x)
\]
where $\theta$ and $x$ lie in bigrading $(0,0)$, and $(0,-1)$, respectively. There are two possible cases that could occur:
\begin{enumerate}
\item[Case 1] $d_2 (\theta ) = 0$. 
Using the fact that the differentials $d_k$ are $Q$-module maps, it follows that $d_k = 0$ for all $k \geq 2$. We have $q_\infty = 0$.

\item[Case 2] $d_2 (\theta ) \neq 0$. In this case $E_{3}^{\ast , even} =  Q \langle \theta_1 , \theta_0  \rangle / (U \theta_0 - p \theta_1 , Sy )$ where $\theta_1 = [ U \theta ]$, $\theta_0 = [p\theta ]$, and $E_{3}^{\ast , odd} = Q \langle x \rangle/(Ux , Sx )$. For grading reasons, $d_k = 0$ for all $k \geq 3$. We have $q_\infty = 1$.
\end{enumerate}
\end{example}

\section{The family mixed map without family spin-c structure}

In this section, we define the Bauer--Furuta invariant for smooth families of 4-manifolds which do not necessarily carry family spin-c structures. Such generality is needed in the proof of Theorem \ref{thm: main}.

\subsection{The Seiberg--Witten map for 4-manifolds with boundary}
We start by recalling the definition of the Seiberg--Witten equations on a 4-manifold with boundary, following \cite{ManolescuStablehomotopytype,KhandhawitNewGaugeSlice}. This is an important step in the definition of relative Bauer-Furuta invariant.

Let $M$ be smooth, oriented, compact 4-manifold with single boundary component $Y$, which we assume is a rational homology 3-sphere. Let $\fs_{M}$ be a spin-c structure on $M$ that restricts to a spin-c structure $\fs_Y$ on $Y$. We fix a metric on $M$ which is cylindrical near the boundary. We use $S_{M} = S^{+}_{M}\oplus S^{-}_{M}$, $S_{Y}$ to denote the spinor bundles on $M$, $Y$, respectively. Given a spin-c connection $A$, we use $A^t$ to denote the induced connection on the determinant line bundle of the spinor bundle.

Since $Y$ is a rational homology 3-sphere, there exists a spin-c connection $A_0$ on $S_{Y}$ such that $A^t$ is flat. Such $A_0$ is unique up to gauge transformations. Fix an integer $k\geq 3$. The Coulomb slice $\coul(Y,A_0)$ is defined as the $L^{2}_{k}$-Sobolev completion of the space 
\[
\{(A_{0}+a,\psi)\mid a\in \Omega^{1}(Y,i\mathbb{R}),\ \psi\in \Gamma(S_{Y}),\  d^*a=0 \}.
\]

Next, we consider the following space of spin-c connections on $M$ 
\begin{equation}\label{eq: harmonic curvature}
\mathcal{A}:=\{A\mid d^{*}F_{A^{t}}=0,\  F_{\mathbf{t}(A)^{t}}=0\}.    
\end{equation}
Here $\textbf{t}(A)$ denotes the pull back of $A$ to  $Y$. We define the Picard torus $\Pic(M,\fs_{M})$ as the quotient of $\mathcal{A}$ by the gauge group $\mathcal{G}(M)$. There is a noncanonical homeomorphism 
\[
\Pic(M,\fs_{M})\cong T^{b_{1}(M)}
\]
defined by taking the holonomy along a collection of loops that represents a basis of $H_{1}(M;\mathbb{R})$. We have a family of Dirac operator
\[D^{+}: (\mathcal{A}\times S^{+}_{M})/\mathcal{G}(M)\to (\mathcal{A}\times S^{-}_{M})/\mathcal{G}(M)
\]
parameterized by $\Pic(M,\fs_{M})$.

Let $\fs'_M$ be another spin-c structure that is isomorphic to $\fs_{M}$. We pick any isomorphism $\rho: S_{M}\to S'_{M}$ between the spinor bundles and we define the homeomorphism 
\[
\Pic(M,\fs_M)\xrightarrow{\cong} \Pic(M,\fs'_M)
\]
by sending $[A]$ to $[\rho(A)]$. This homeomorphism is \emph{canonical} and is independent with $\rho$. We will use this homeomorphism to identify Picard tori for isomorphic spin-c structures.

Now we pick a smooth base connection $\hat{A}_0$ that lies in (\ref{eq: harmonic curvature}). The double coulomb slice $\coul^{cc}(M,\hat{A}_0)$ is defined as the  $L^{2}_{k+1/2}$-Sobolev completion of the space 
\[
\{(\hat{A}_{0}+a,\psi)\mid a\in \Omega^{1}(M;i\mathbb{R}),\ \psi\in \Gamma(S_{M}^+),\ d^*\alpha=0,\ d^*(\mathbf{t}(\alpha))=0\}.
\]
The slice $\coul^{cc}(M,\hat{A}_0)$ 
is not preserved by the full gauge group $\mathcal{G}_{M}$. Instead, it is preserved by the harmonic, based gauge group 
\[
\mathcal{G}^{h}(M):=\{g: M\to S^1\mid g \text{ is harmonic, and }g|_{Y}=1\}\cong H^{1}(M;\mathbb{Z}).
\]
Set the quotient Coulomb slice \[\mathcal{B}^+(M,\hat{A}_{0}):=\coul^{cc}(M,\hat{A}_0)/\mathcal{G}^{h}(M).\] 
We have a Hilbert bundle 
\[
\operatorname{pj}^{+}:\mathcal{B}^+(M,\hat{A}_0)\to \Pic(M,\fs_M)
\]
defined by 
\[
[(\hat{A}_0+a,\psi)]\mapsto [\hat{A}_0+p^{\perp}(a)].
\]
Here $p^{\perp}(a)$ denotes the orthogonal projection to the space 
\[
\mathcal{H}:=\{a\in \Omega^{1}(M;i\mathbb{R})\mid da=d^*a=0,\ \mathbf{t}(a)=0\}\cong H^{1}(M;\mathbb{R}).
\]
To
define the Seiberg--Witten map, we introduce another Hilbert bundle \[\operatorname{pj}^{-}: \mathcal{B}^{-}(M,\hat{A}_0)\to \Pic(M,\fs_M)\] 
as follows. Let $\mathcal{V}^-$ be the $L^{2}_{k-1/2}$ completion of the space 
\[
\{(\hat{A}_{0}+a,b,\psi)\mid a\in \mathcal{H},\ b\in \Omega^{2}_{+}(M;i\mathbb{R}),\ \psi\in \Gamma(S^-_{M})\}.
\]
We define $\mathcal{B}^{-}(M,\hat{A}_0)$ as the quotient $\mathcal{V}^-$ by the group $\mathcal{G}^h(M)$, which acts trivially on the $b$-component. And we define the map $\operatorname{pj}^{-}$ by 
\[
\operatorname{pj}^{-}([\hat{A}_{0}+a,b,\psi])=[\hat{A}_{0}+a].
\]
Now consider the Seiberg--Witten map $\overline{SW}:\coul^{cc}(M,A_M)\to \mathcal{V}^{-}$
defined by 
\[
\overline{SW}(\hat{A}_0+a,\psi)=(\hat{A}_0+p^{\perp}(a),F^{+}_{A^{t}}-\rho^{-1}(\psi\psi^*)_{0}, D_{\hat{A}_0+a}\psi).
\]
This map is equivariant under the action of $\mathcal{G}^h(M)$ hence induces the Seiberg--Witten map 
\[
SW: \mathcal{B}^{+}(M,\hat{A}_0)\to \mathcal{B}^{-}(M,\hat{A}_0),
\]
that covers the identity map of $\Pic(M,\fs_M)$. 

Let $A_0=\mathbf{t}(\hat{A}_0)$. Since all gauge transformations in $\mathcal{G}^{h}(M)$ are trivial on $Y$, we have a well-defined restriction map 
\[
r:\mathcal{B}^{+}(M,\hat{A}_0)\to \coul(Y,A_0)
\]
defined by 
\[
r([(\hat{A}_0+a,\psi)])=[(A_0+\mathbf{t}(a),\psi|_{Y})]. 
\]
By doing finite-dimensional approximations on the $S^1$-equivariant map of Hilbert bundles over $\mathrm{Pic}(M, \mathfrak{s}_M )$ given by
\[
(SW,r):\mathcal{B}^{+}(M,\hat{A}_0)\to \mathcal{B}^{-}(M,\hat{A}_0)\times \coul(Y,A_0),
\] one obtains the relative Bauer--Furuta invariant $\Bafu(M,\fs_{M})$. We will explain this construction in the family setting. For this purpose, we need to discuss the naturality of $(SW,r)$. Let $\fs'_{M}$ be another spin-c structure on $M$ such that   
\[
\fs'_{M}\cong \fs_{M} \text{ and }\fs'_{M}|_{Y}=\fs_{M}|_{Y}.
\]
Let $S'_{M}$ be the spinor bundle for $\fs'_{M}$. Then we can canonically identify $S'_{M}|_{Y}$ with $S_{M}|_{Y}$. Let $\hat{A}'_0$ be a spin-c connection on $S'_{M}$ such that $F_{(\hat{A}'_0)^{t}}$ is harmonic and $\hat{A}'_{0}|_{Y} = A_0$.
\begin{lemma}\label{lem: gluing spinor bundles}
 There exists an isomorphism 
$\tau: S_{M}\xrightarrow{\cong} S'_{M}$ that equals the identity on $S_{M}|_{Y}$ and satisfies $\tau^{*}(\hat{A}'_{0}) \in \hat{A}_{0}+\mathcal{H}$.    
Moreover, any two such $\tau$ differ by the composition with an element of $\mathcal{G}^{h}(M)$.
\end{lemma}
\begin{proof}   
We start with an arbitrary isomorphism $\tau_0:S_{M} \xrightarrow{\cong}S'_{M}$. Since the gauge group $\mathcal{G}(Y)$ is connected, we can modify $\tau_0$ near $\partial M$ such that $\tau_0$ equals identity on $Y$. Let 
\[
a=\tau_0^*(\hat{A}'_{0})-\hat{A}_{0}\in i\Omega^{1}(M).
\]
Then we have $da=0$ since 
\[
F_{(\tau_0^*(\hat{A}'_{0}))^{t}}=F_{\hat{A}_{0}^{',t}}=F_{\hat{A}_{0}^{t}}.
\]
Consider the relative de-Rham complex 
\[
0\to \Omega^{0}(M,Y)\to \Omega^{1}(M,Y)\to \Omega^{2}(M,Y)\to \cdots,
\]
where $\Omega^{i}(M,Y)$ denotes the space of $i$-forms on $M$ that pulls back to $0$ on $Y$. The homology of this complex computes $H^{*}(M;Y;\mathbb{R})$ and any element is uniquely presented by a harmonic form in $\Omega^*(M,Y)$. Therefore, there exists a unique $h\in \mathcal{H}$ such that $h-a=id\theta$ for some $\theta\in \Omega^{0}(M,Y)$. We let $g=e^{i\theta}\in \mathcal{G}(M)$. Then $\tau=g\circ \tau_0$  satisfies the requirement. 

Given any two choices $\tau$ and $\tau'$, they differ by an element $g:M\to S^{1}$ that satisfies 
\[
g|_{Y}=1, \quad g(\hat{A}'_{0})-\hat{A}'_{0}\in \mathcal{H}.
\]
This is equivalent to $g\in \mathcal{G}^{h}(M)$.
\end{proof}

We choose an isomorphism $\tau: S_{M}\to S'_{M}$ from Lemma \ref{lem: gluing spinor bundles} and use it to define the homeomorphism 
\[
\tau^{\pm}_{\hat{A}_0,\hat{A}'_0}:\mathcal{B}^{\pm}(M,\hat{A}_0)\xrightarrow{\cong} \mathcal{B}^{\pm}(M,\hat{A}'_0)
\]
that covers the identity map of $\Pic(M,\fs_{M})$. Note that the choice of $\tau$ does not affect the map $\tau^{\pm}_{\hat{A}_0,\hat{A}'_0}$ because we already quotient out by the action of $\mathcal{G}^{h}(M)$. Furthermore, by the uniqueness of $\tau^{\pm}_{*,*}$, for any base connections $\hat{A}_0, \hat{A}'_0$ and $\hat{A}''_0$, we have 
\begin{equation}\label{eq: gluing compatible}
\tau^{\pm}_{\hat{A}_0,\hat{A}''_0}=\tau^{\pm}_{\hat{A}'_0,\hat{A}''_0}\circ \tau^{\pm}_{\hat{A}_0,\hat{A}'_0}.    
\end{equation}

By gauge equivalence of the Seiberg--Witten map, the following diagram commutes
\begin{equation}\label{diag: gluing SW}
\xymatrix{\mathcal{B}^{+}(M,\hat{A}_0)\ar[rr]^-{(SW,r)}\ar[d]^{\tau^{+}_{\hat{A}_0,\hat{A}'_0}} && \mathcal{B}^{-}(M,\hat{A}_0)\times \coul(Y,A_0)\ar[d]^{(\tau^{-}_{\hat{A}_0,\hat{A}'_0},\operatorname{id})} \\
\mathcal{B}^{+}(M,\hat{A}'_0)\ar[rr]^-{(SW,r)} && \mathcal{B}^{-}(M,\hat{A}'_0)\times \coul(Y,A_0).
}    
\end{equation}

\subsection{Relative Bauer--Furuta invariants of smooth families}
Now we define the relative Bauer--Furuta invariant of smooth families of 4-manifolds.

We focus on the case that is considered in our application. Let $Y$ be a Seifert rational homology 3-sphere with a spin-c structure structure $\fs_Y$. As before, we consider the Seifert $G$-action on $Y$. Further, we fix a lifting of this $G$-action to the spinor bundle $S_{Y}$. The Borel construction gives the smooth family 
\begin{equation}\label{eq: Borel Y} 
Y\to (Y\times EG)/G\to BG    
\end{equation}
with a family spin-c structure $\widetilde{\fs}$. 
Let $B$ be a finite CW complex with a continuous map $B\to BG$. By pulling back the family (\ref{eq: Borel Y}) and the family spin-c structure $\widetilde{\fs}$, one obtains a smooth family $Y\to \widetilde{Y}\to B$
and a family spin-c structure $\fs_{\widetilde{Y}}$ on it. 

Let $(M,\fs_{M})$ be a smooth spin-c 4-manifold bounded by $(Y,\fs)$, and we let 
\[M\to \widetilde{M}\to B\]
be a smooth family which restricts to the family $\widetilde{Y}\to B$. We do not assume that $\widetilde{M}\to B$ admits a family spin-c structure whose fiber restriction equals $\fs_{M}$. Instead, we just assume that the monodromy of $\widetilde{M}\to B$ preserves the isomorphism class of $\fs_{M}$. We use $M_{b}$ (resp. $Y_b$) to denote the fiber of $\widetilde{M}/B$ (resp. $\widetilde{Y}/B$) over $b\in B$. 

Associated to the family $\widetilde{M}/B$, we have the family Picard torus \[\Pic(\widetilde{M}/B,\fs_{M})=\bigcup_{b\in B}\Pic(M_{b},\fs_{M}).\] This is a (possibly nontrivial) principal $T^{b_{1}(M)}$-bundle over $B$. 

By choosing a $G$-invariant metric $g$ on $Y$, we obtain a fiberwise metric $g_{\widetilde{Y}}$ on $\widetilde{Y}/B$. Then we extend it to a fiberwise metric $g_{\widetilde{M}}$ on $\widetilde{M}/B$. Let $A_0$ be a spin-c connection on $Y$ with $F_{A^t_0}=0$. By taking the average in the affine space of spin-c connections 
, we may assume that $A_0$ is $G$-invariant. Then $A_0$ can be regarded as a spin-c connection on $Y_b$, denoted as $A_{0,b}$. We define the family Coulomb slice 
\[
\coul(\widetilde{Y}/B,A_0)=\bigcup_{b\in B}\coul(Y_{b},A_{0,b}),
\]
which is a Hilbert bundle over $B$.

Now we take a finite open cover $\{U_\alpha\}$ of $B$ such that each $U_\alpha$ is contractible. We use $\widetilde{M}_{\alpha}/U_{\alpha}$ to denote the restriction of $\widetilde{M}/B$. Then on each $\widetilde{M}_{\alpha}/U_\alpha$, there exists a family spin-c structure $\fs_{\widetilde{M}_{U_\alpha}}$ whose boundary restriction equals $\fs_{\widetilde{Y}}$ and  whose fiber restriction is isomorphic to $\fs_{M}$. 

\begin{lemma}\label{lem: harmonic family base connection}
For each $\alpha$, there exists a smooth fiberwise spin-c connections $\tilde{A}_{0,\alpha}=\{\hat{A}_{0,b,\alpha}\}_{b\in U_\alpha}$ for $\fs_{\widetilde{M}_{U_\alpha}}$ such that for any $b\in U_{\alpha}$, we have $\hat{A}_{0,b,\alpha}|_{\partial M_{b}}= A_{0,b}$ and $F_{\hat{A}^{t}_{0,b,\alpha}}$ is harmonic.
\end{lemma}
\begin{proof} We fist extend $\{A_{0,b}\}_{b\in U_{\alpha}}$ to a smooth fiberwise spin-c connection $\{\hat{A}'_{0,0,b}\}_{b\in U_{\alpha}}$. Then we have decomposition 
\[
F_{(\hat{A}'_{0,b,\alpha})^{t}}=\beta_{b}+h_{b}
\]
where $\beta_{b}\in \operatorname{Im}(d)$ and $h_{b}$ is harmonic. Let $\gamma_{b}\in \Omega^{1}(M_{b};i\mathbb{R})$ be the unique 1-form that satisfies 
\[
\gamma_{b}\in \operatorname{Im}(d^*:\Omega^{2}(M;i\mathbb{R})\to \Omega^{1}(M;i\mathbb{R}))\text{ and } d^{*}(\gamma_{b})=\frac{\beta_{b}}{2}.
\]
Then the fiberwise connection 
\[
\{\hat{A}_{(0,b,\alpha)}:=\hat{A}'_{(0,b,\alpha)}-\gamma_{b}\}_{b\in U_{\alpha}}
\]
satisfies the requirements.
\end{proof}
We pick $\tilde{A}_{0,\alpha}$ from Lemma \ref{lem: harmonic family base connection}. Repeating the construction in the last section, we obtain spaces
\[
\mathcal{B}^{\pm}(\widetilde{M}_{\alpha},\tilde{A}_{0,\alpha}):=\bigcup_{b\in U_{\alpha}}\mathcal{B}^{\pm}(M_{b},\hat{A}_{0,b,\alpha}),
\]
which are Hilbert bundles over $\Pic(\widetilde{M}/B,\fs_{M})|_{U_{\alpha}}$. 

For each $b\in U_{\alpha}\cap U_{\beta}$, we have the homeomorphism
\[
\tau^{\pm}_{\hat{A}_{0,b,\alpha},\hat{A}_{0,b,\beta}}: \mathcal{B}^{\pm}(M_{b},\hat{A}_{0,b,\alpha})\xrightarrow{\cong}\mathcal{B}^{\pm}(M_{b},\hat{A}_{0,b,\beta})
\]
from Lemma \ref{lem: gluing spinor bundles}. Since these homeomorphisms satisfy the compatibility condition  (\ref{eq: gluing compatible}), we can glue together $\{\mathcal{B}^{\pm}(\widetilde{M}_{\alpha},\tilde{A}_{0,\alpha})\}_{\alpha}$ and form a Hilbert bundle \[\mathcal{B}^{\pm}(\widetilde{M}/B,\fs_{M})\to \Pic(\widetilde{M}/B,\fs_{M}).\]
Moreover, by the diagram (\ref{diag: gluing SW}), we can glue together the Seiberg--Witten maps and the restriction maps to obtain the $S^1$-equivariant maps
\begin{equation}\label{eq: family SW}
\widetilde{SW}: \mathcal{B}^{+}(\widetilde{M}/B,\fs_{M})\to \mathcal{B}^{-}(\widetilde{M}/B,\fs_{M})    
\end{equation}
and 
\begin{equation}
\widetilde{r}:     \mathcal{B}^{+}(\widetilde{M}/B,\fs_{M})\to \coul(\widetilde{Y}/B,A_0).
\end{equation}

The next step is to do finite-dimensional approximations. This step follows closely with previous works \cite{ManolescuStablehomotopytype,Baraglia2024equivariant,KangParkTaniguchi} so we only give a sketch here. First, we use Kuiper's theorem to get a trivialization of the Hilbert bundle
\[
\mathcal{B}^{-}(\widetilde{M}/B,\fs_{M})\cong \Pic(\widetilde{M}/B,\fs_{M})\times \mathcal{U}^{-}.
\]
We compose (\ref{eq: family SW}) with the projection $\mathcal{U}^{-}$ and obtain 
\[
\widetilde{SW}'=\widetilde{l}+\widetilde{q}: \mathcal{B}^{+}(\widetilde{M}/B,\fs_{M})\to \mathcal{U}^{-}. 
\]
Here we decompose $\widetilde{SW}'$ into the linear part $\widetilde{l}$ and the nonlinear part $\widetilde{q}$.  Next, for $\lambda < \mu$, we let $V_{\lambda}^{\mu}\subset \coul(Y,A_0)$ be the sum  of eigenspacses of the elliptic operator $(D,d^*)$ with eigenvalue in $(\lambda,\mu]$. Then $V_{\lambda}^{\mu}$ is invariant under the $G$-action. Let 
\[\widetilde{V}_{\lambda}^{\mu}:=V^{\mu}_{\lambda}\times_{G}E,\]
where $E=B\times_{BG}EG$. We pick a finite-dimensional subspace $U^{-}\subset \mathcal{U}^{-}$. Set 
\[
U^{+}=\widetilde{l}^{-1}(U^{-})\cap (p^{\mu}_{-\infty}\circ \widetilde{r})^{-1}(0) 
\]
Here $p^{\mu}_{-\infty}: \coul(\widetilde{Y}/B,A_0)\to \widetilde{V}^{\mu}_{-\infty}$ denotes the fiberwise orthogonal projection. Consider the approximated Seiberg--Witten map 
\[
\widetilde{SW}_{\operatorname{apr}}=\widetilde{l}+\operatorname{pj}\circ \widetilde{q}: U^{+}\to U^{-},
\]
where $\operatorname{pj}$ denotes the orthogonal projection to $U^-$. Consider the map 
\begin{equation}\label{eq: approximated SW with restriction map}
(\widetilde{SW}_{\operatorname{apr}}, p^{\mu}_{\lambda}\circ \widetilde{r}): U^{+}\to U^{-}\times \widetilde{V}^{\mu}_{\lambda}.    
\end{equation}

We pick large enough $U^{-}$, $\lambda\ll 0 \ll \mu$ and $0\ll R$, $0<\epsilon\ll 1$. Then by the compactness property of the Seiberg--Witten equations, there exists a $G\times S^1$-equivariant Conley index pair $(K,L)$ of the approximated Chern-Simons-Dirac flow on $V^{\mu}_{\lambda}$ such that (\ref{eq: approximated SW with restriction map}) induces a well-defined $S^{1}$-equivariant map 
\begin{equation}\label{eq: unsuspended relative BF}
B(U^{+},R)/S(U^{+},R)\to B(U^-,\epsilon)/S(U^-,\epsilon)\wedge (K/L\wedge E_{+})/G.     
\end{equation}
Here $B(-,-)$ and $S(-,-)$ denote the disk/sphere (bundles) of a given radius. We consider the induced map of (\ref{eq: unsuspended relative BF}) on $c\tilde{H}^{*}_{S^1}(-)$ and $t\tilde{H}^{*}_{S^1}(-)$. Since $B(U^{+},R)/S(U^{+},R)$ is the Thom space of an oriented $S^1$-vector bundle over the Picard torus $\Pic(\widetilde{M}/B,\fs_{M})$, we have the isomorphisms 
\[
\begin{split}
c\tilde{H}^{*}_{S^{1}}(B(U^{+},R)/S(U^{+},R);\F)\cong c\tilde{H}^{*-\operatorname{dim}(U_{+})}_{S^{1}}(\Pic(\widetilde{M}/B,\fs_{M});\F)\\
t\tilde{H}^{*}_{S^{1}}(B(U^{+},R)/S(U^{+},R);\F)\cong t\tilde{H}^{*-\operatorname{dim}(U_{+})}_{S^{1}}(\Pic(\widetilde{M}/B,\fs_{M});\F)\end{split}.
\]
Moreover, since the $S^1$-action on $\Pic(\widetilde{M}/B,\fs_{M})$ is trivial. We have the isomorphisms 
\[
\begin{split}
c\tilde{H}^{*}_{S^{1}}(\Pic(\widetilde{M}/B,\fs_{M});\F)&\cong H^{*}(\Pic(\widetilde{M}/B,\fs_{M});\F)\otimes_{\F}\F[U,U^{-1}]/(U)\\ t\tilde{H}^{*}_{S^{1}}(\Pic(\widetilde{M}/B,\fs_{M});\F)&\cong H^{*}(\Pic(\widetilde{M}/B,\fs_{M});\F)\otimes_{\F}\F[U,U^{-1}].\end{split}
\]
Combining these isomorphisms with the induced map by (\ref{eq: unsuspended relative BF}), we obtain the maps 
\begin{equation}\label{eq: cobordism map 1}
\begin{split}
c\tilde{H}^{*-\operatorname{dim}(U_{-})}_{S^1}(K/L\wedge E_{+})/G;\F)&\to  H^{*-\operatorname{dim}(U_{+})}(\Pic(\widetilde{M}/B,\fs_{M});\F)\otimes_{\F}\F[U,U^{-1}]/(U)\\
t\tilde{H}^{*-\operatorname{dim}(U_{-})}_{S^1}(K/L\wedge E_{+})/G;\F)&\to  H^{*-\operatorname{dim}(U_{+})}(\Pic(\widetilde{M}/B,\fs_{M});\F)\otimes_{\F}\F[U,U^{-1}]
\end{split}
\end{equation}
The $G$-equivariant map $E\to EG$ induces an $S^{1}$-equivariant map 
\begin{equation}\label{eq: K/L wedge E to EG}
(K/L\wedge E_{+})/G\to (K/L\wedge EG_{+})/G= I^{\mu}_{\lambda}(Y,\fs_{Y},g)/\!\!/ G.    
\end{equation}
By composing (\ref{eq: cobordism map 1}) with the map induced by (\ref{eq: K/L wedge E to EG}), one obtains maps 
\begin{equation}\label{eq: cobordism map 2}
\begin{split}
\widehat{HM}^*_{G}(\widetilde{M}/B,\fs_{M}):\widehat{HM}^{*}_{G}(Y,\fs_{Y};\F)&\to  H^{*}(\Pic(\widetilde{M}/B,\fs_{M});\F)\otimes_{\F}\F[U,U^{-1}]/(U)\\
\overline{HM}^*_{G}(\widetilde{M}/B,\fs_{M}):\overline{HM}^{*}_{G}(Y,\fs_{Y};\F)&\to  H^{*}(\Pic(\widetilde{M}/B,\fs_{M});\F)\otimes_{\F}\F[U,U^{-1}]
\end{split}
\end{equation}
Now we discuss some key properties of these maps.
\begin{itemize}
    \item (\ref{eq: cobordism map 2}) are homogeneous of degree $\frac{c^{2}_{1}(\fs_{Y})-\sigma(M)}{4}-b^{+}(M)$.
    \item Let $Q_{1}=\F[S,R]/(R^2)$ and $Q_{2}=\F[S,R,U]/(R^2)$. Then we can make $H^{*}(\Pic(\widetilde{M}/B,\fs_{M});\F)$ into a $Q_{1}$-module via the map 
    \[
    \Pic(\widetilde{M}/B,\fs_{M})\to B\to BG.
    \]
    Since $\widetilde{SW}$ and $\widetilde{r}$ covers the identity map on $B$, (\ref{eq: cobordism map 1}) are maps between $H^{*}(B;\F)$-modules.
    As a result,  (\ref{eq: cobordism map 2}) are maps between $Q_{2}$-modules. 
    \item By the natural transformation $c\tilde{H}^{*}_{S^{1}}(-)\to t\tilde{H}^{*}_{S^{1}}(-)$, we have the commutative diagram 
\begin{equation}\label{diagram: HM-bar to HM-hat}
\xymatrix{\overline{HM}^{*}_{G}(Y,\fs_{Y};\F)\ar[d]^{k}\ar[rr]^-{\overline{HM}^*_{G}(\widetilde{M}/B,\fs_{M})} && H^{*}(\Pic(\widetilde{M}/B,\fs_{M});\F)\otimes_{\F}\F[U,U^{-1}]\ar[d]^{k'}
\\
\widehat{HM}^{*}_{G}(Y,\fs_{Y};\F)\ar[rr]^-{\widehat{HM}^*_{G}(\widetilde{M}/B,\fs_{M})} && H^{*}(\Pic(\widetilde{M}/B,\fs_{M});\F)\otimes_{\F}\F[U,U^{-1}]/(U)
}    
\end{equation}
Here $k$ is the map in the exact triangle (\ref{eq: exact triangle 1}) and $k'$ is the quotient map.
\item Since the $S^1$-actions on both the domain and target of (\ref{eq: unsuspended relative BF}) are semi-free, the map $\overline{HM}^*_{G}(\widetilde{M}/B,\fs_{M})$ is determined by the restriction of (\ref{eq: approximated SW with restriction map}) to the $S^{1}$-fixed sets, denoted by
\begin{equation}
(\widetilde{SW}_{\operatorname{apr}}^{S^1}, (p^{\mu}_{\lambda}\circ \widetilde{r})^{S^1}): (U^{+})^{S^{1}}\to (U^{-})^{S^{1}}\times (\widetilde{V}^{\mu}_{\lambda})^{S^1}.    
\end{equation}
Note that the map $\widetilde{SW}_{\operatorname{apr}}^{S^{1}}$ is a linear because the quadratic term $\widetilde{q}$ vanishes on reducible configurations. Now suppose $b^{+}(M)>0$. Then $\widetilde{SW}_{\operatorname{apr}}^{S^{1}}$ is not subjective. We denote its image by $V^{-}\subsetneq (U^{-})^{S^1}$. Then restriction of (\ref{eq: unsuspended relative BF}) to the $S^{1}$-fixed sets factors through the inclusion map 
\[
\begin{split}
B(V^{-},\epsilon)/S(V^{-},\epsilon) \wedge (K^{S^{1}}/L^{S^{1}}\wedge E_{+})/G&\longrightarrow \\ B((U^{-})^{S^1},\epsilon)/S((U^{-})^{S^1},\epsilon)&\wedge (K^{S^{1}}/L^{S^{1}}\wedge E_{+})/G.
\end{split}
\]
Since this map is null-homotopic, we see that $\overline{HM}^*_{G}(\widetilde{M}/B,\fs_{M})=0$ when $b^{+}(M)>0$. By (\ref{diagram: HM-bar to HM-hat}), the map $\widehat{HM}^*_{G}(\widetilde{M}/B,\fs_{M})$ induces a ``mixed map''
\[
\overrightarrow{HM}^*_{G}(\widetilde{M}/B,\fs_{M}):HM^{*}_{\reduced,G}(Y,\fs_{Y};\F)\to H^{*}(\Pic(\widetilde{M}/B,\fs_{M});\F)\otimes_{\F}\F[U,U^{-1}]/(U).
\]
This is a map between $Q_{2}$-modules. It fits into the following diagram
\begin{equation}\label{diag: HM-arrow natural 2}
\xymatrix{HM^*_{\reduced,G}(Y,\fs_{Y};\F)\ar[d]^{q^*}\ar[rrr]^-{\overrightarrow{HM}^*(\widetilde{M}/B,\fs_{M})}& & &H^*(\Pic(\widetilde{M}/B,\fs_{M});\F)\otimes_{\F}\F[U,U^{-1}]/(U)\ar[d]^{j^*\otimes \operatorname{Id}}\\
HM^*_{\reduced}(Y,\fs_{Y};\F)\ar[rrr]^-{\overrightarrow{HM}^*(M,\fs_{M})}& & &H^*(\Pic(M,\fs_{M});\F)\otimes_{\F}\F[U,U^{-1}]/(U)
} \end{equation}
Here $j^*$ is induced by the fiber inclusion \[j: \Pic(M,\fs_{M})\to \Pic(\widetilde{M}/B,\fs_{M}).\] And $q^*$ is defined in (\ref{eq: equivariant red to red}). And 
\begin{equation}\label{eq: mixed HM}
\overrightarrow{HM}^*(M,\fs_{M}):HM^*_{\reduced}(Y,\fs_{Y};\F)\to H^*(\Pic(M,\fs_{M});\F)\otimes_{\F}\F[U,U^{-1}]/(U)    
\end{equation}
is the mixed map for the single manifold $M$. 
\end{itemize}
\begin{remark} We expect that all these family cobordism maps are independent of auxiliary choices (e.g. the family metric on $\widetilde{M}/B$, the numbers $\mu,\lambda$ and the subspace $U^{-}$). However, such invariance is not needed in our proof. What we use is just the existence of a $Q_{2}$-module map that fits into (\ref{diag: HM-arrow natural 2}). Also note that the map $\overrightarrow{HM}^*(M,\fs_{M})$ is induced by the relative Bauer--Furuta invariant for the single spin-c 4-manifold $(M,\fs_{M})$. So the invariance of  $\overrightarrow{HM}^*(M,\fs_{M})$ follows from the invariance of $\operatorname{BF}(M,\fs_{M})$, which is proved in \cite{ManolescuStablehomotopytype}. We emphasize that the map $\overrightarrow{HM}^*(M,\fs_{M})$ above denotes the map defined via the relative Bauer-Furuta invariant, not the map defined in \cite{KronheimerMonopoles} via counting monopoles on $M\cup([0,\infty)\times Y)$.
In particular, we do not use the expectation (which is not proven yet) that the cobordism map between monopole Floer homologies in \cite{KronheimerMonopoles} coincides with the map induced from the relative Bauer--Furuta invariant.  
\end{remark}

\section{The contact invariant in $G$-equivariant monopole Floer cohomology}
Given a contact structure $\xi$ on a rational homology 3-sphere $Y$ with spin-c structure $\fs_{Y}$. Both Iida--Taniguchi \cite{IidaTaniguchi} and Roso \cite{Roso} defined a Floer homotopy version of contact invariant\footnote{There is a sign ambiguity in the definition of $\Psi(Y,\xi)$, which is unavoidable \cite{LRSMonopoleLefschetz}. This sign ambiguity will not affect our argument.}. For our purpose, we use Iida--Taniguchi's version, which is a \emph{nonequivariant} stable map 
\[
\Psi(Y,\xi): S^0\to \Sigma^{\frac{1}{2}+d_{3}(Y,[\xi])}\SWF(-Y,\fs_{Y}).
\]
Here $d_{3}(Y,[\xi])$ denotes the Gompf $d_3$-invariant for the 2-plane field $[\xi]$, defined as 
\[
d_{3}(Y,[\xi])=\frac{1}{4}(c^2_{1}(X)-2\chi(X)-3\sigma(X))
\]
for any almost complex bounding  $X$ of $(Y,\xi)$. By the duality between $\SWF(-Y,\fs_{Y})$ and $\SWF(Y,\fs_{Y})$, one can also write $\Psi(Y,\xi)$ as 
\begin{equation}\label{eq: homotopy contact invariant}
\Psi(Y,\xi)^*: \Sigma^{-\frac{1}{2}-d_{3}(Y,[\xi])}\SWF(Y,\fs_{Y})\to S^0
\end{equation}

To define this invariant, the authors did finite-dimensional approximation of the Seiberg--Witten equations on the symplectization $(\mathbb{R}^{\geq 1}\times Y, \frac{1}{2}d(t^2\theta))$, where $\theta$ is the contact form, and obtain the map  
\[
\Psi:S^{V_{+}}\wedge I^{\mu}_{\lambda}(Y,\fs_{Y},g)\to S^{V_{-}}    
\]
for suitable vector spaces $V_{\pm}$.

One advantage of the approach by a finite-dimensional approximation is that it does not require transversality. Thus, one can define an equivariant version of the invariant for a contact 3-manifold with a group action, which has been discussed by Roso~\cite{roso2023contactmoduloplspace} and Iida--Taniguchi~\cite{iida2024monopolestransverseknots}. We also need the equivariant version, summarized as follows.

Suppose $\xi$ is invariant under a $G=\mathbb{Z}/p$ action on $Y$. Then one can do finite-dimensional approximations $G$-equivariantly and make $\Psi$ into a $G$-equivariant map. (In this case, $V_{\pm}$ are suitable $G$-representation spaces.) Consider the induced map 
\begin{equation}\label{eq: contact induced map}
\begin{split}
H^{*}(BG;\F)&\cong \tilde{H}^{*+\operatorname{dim}(V_{-})}((S^{V_{-}}\wedge EG_{+})/G;\F)\\
&\xrightarrow{\Psi} \tilde{H}^{*+\operatorname{dim}(V_{-})}((S^{V_{+}}\wedge I^{\mu}_{\lambda}(Y,\fs_{Y},g) \wedge EG_{+})/G;\F) \\ &\cong \tilde{H}^{*+\operatorname{dim}(V_{-})-\operatorname{dim}(V_{+})}((I^{\mu}_{\lambda}(Y,\fs_{Y},g) \wedge EG_{+})/G;\F)\\
&\cong \widetilde{HM}^{*\frac{1}{2}+d_{3}(Y,[\xi])}_{G}(Y,\fs_{Y}).
\end{split}    
\end{equation}
We define the equivariant contact invariant 
\[
\widetilde{\psi}_{G}(Y,\xi)\in \widetilde{HM}^{*}_{G}(Y,\fs_{Y})
\]
as the image of $1\in H^*(BG;\F)$ under (\ref{eq: contact induced map}). We also define hat version 
\[
\widehat{\psi}_{G}(Y,\xi):= p(\widetilde{\psi}_{G}(Y,\xi))\in \widehat{HM}^*_{G}(Y,\fs_{Y};\F),
\]
 the reduced version 
\[
\psi_{G,\reduced}(Y,\xi):=k(\widehat{\psi}_{G}(Y,\xi))\in HM^*_{G,\reduced}(Y,\fs_{Y};\F),
\]
where $p$ and $k$ are maps in the exact triangle (\ref{eq: exact triangle 3}). Under the induced map by the fiber $q$ in (\ref{eq: SWF fibration}), these $G$-equivariant contact invariant pulls back to their nonequivariant counterparts. 
\[
\begin{split}
\widetilde{\psi}(Y,\xi)&:=q^*(\widetilde{\psi}_{G}(Y,\xi))\in \widetilde{HM}^*(Y,\fs_{Y};\F)\\
\psi(Y,\xi)&:=q^*(\psi_{G}(Y,\xi))\in \widehat{HM}^*(Y,\fs_{Y};\F) \\
\psi_{\reduced}(Y,\xi)&:=q^*(\psi_{G,\reduced}(Y,\xi))\in HM^*_{\reduced}(Y,\fs_{Y};\F)
\end{split}
\]
These nonequivariant versions have been defined and studied in \cite{IidaTaniguchi}.

Let $(M,\omega)$ be a symplectic filling of $(Y,\xi)$ with canonical spin-c structure $\fs_{M}$. We consider the nonequivariant relative Bauer-Furuta invariant as a stable map 
\begin{equation}\label{eq: Bauer-Furuta of symplectic filling}
\Bafu(M,\fs_{M}): \operatorname{Pic}(M,\fs_{M})_{+}\to \Sigma^{b^{+}(M)-\frac{c^2_{1}(\fs_{M})-\sigma(M)}{4}}\SWF(Y,\fs_{Y})
\end{equation}
Note that we have 
\[b^{+}(M)-\frac{c^2_{1}(\fs_{M})-\sigma(M)}{4}= -d_{3}(Y,\xi)-\frac{1}{2}+b_{1}(M).\]
Therefore, the composition of (\ref{eq: homotopy contact invariant}) and (\ref{eq: Bauer-Furuta of symplectic filling}) gives the stable map 
\begin{equation}\label{eq: composition of contact invariant with relative BF}
\Psi(Y,\xi)^{*}\circ \Bafu(M,\fs_{M}): \operatorname{Pic}(M,\fs_{M})_{+}\to S^{b_{1}(M)}.    
\end{equation}
\begin{lemma}\label{lem: mapping degree}
The mapping degree of (\ref{eq: composition of contact invariant with relative BF}) equals $\pm 1$.    
\end{lemma}
\begin{proof}
When $b_{1}(M)=0$, this is proved as \cite[Corollary 1.3]{IidaTaniguchi} by Iida--Taniguchi. We now adapt the proof to the case that $b_{1}(M)>0$ as follows.

First, recall that Iida~\cite{ Iida} defined the stable cohomotopy version $\Psi(M,\xi,\fs)$ of Kronheimer--Mrowka's invariant $\mathfrak{m}(M,\xi,\fs)$ \cite{KMcontact} of a spin-c 4-manifold $(M,\fs)$ with contact boundary $(Y, \xi)$.
While $M$ was assumed to have vanishing $b_1$ in \cite{ Iida},
one can generalize Iida's construction to $b_1(M)>0$ by considering a family over the Picard torus $\operatorname{Pic}(M,\fs)$ as for the definition of the Bauer--Furuta invariant \cite{BauerFurutaI} for $b_1>0$. (See Remark \ref{rmk: weighted harmonic}.)
In this way, $\Psi(M,\xi,\fs)$ is formulated as a stable map
\[
\Psi(M,\xi,\fs) : \operatorname{Pic}(M,\fs)_+ \to S^{-d(M,\xi,\fs)},
\]
where $d(M,\xi,\fs)$ is the formal dimension of the moduli space in \cite{KMcontact}.
Iida \cite{Iida} also proved that the mapping degree of $\Psi(M,\xi,\fs)$ is given by the Kronheimer--Mrowka invariant $\mathfrak{m}(M,\xi,\fs)$.
This is also generalized to the case that $b_1(M)>0$ as in the proof of the fact that the Bauer-Furuta invariant recovers the Seiberg--Witten invariant for $b_1>0$ \cite{BauerFurutaI}.
On the other hand, for a symplectic filling $(M,\omega)$ of $(Y, \xi)$, the invariant $\mathfrak{m}(M,\xi,\fs)$ is $\pm1$ \cite{KMcontact}.
Thus the mapping degree of $\Psi(M,\xi,\fs_M)$ is $\pm1$ for the canonical spin-c structure $\fs_M$, regardless of $b_1(M)$.

The result \cite[Corollary 1.3]{IidaTaniguchi} for $b_1(M)=0$ is an immediate consequence of a gluing result \cite[Theorem 1.2]{IidaTaniguchi},
\[
\eta \circ (\Bafu(M,\fs_M) \wedge \Psi(Y,\xi))
= \Psi(M,\xi,\fs_M),
\]
together with the above result that the mapping degree of $\Psi(M,\xi,\fs_M)$ is $\pm1$.
This gluing is also generalized to $b_1(M)>0$ by following \cite{Manolescugluing,KhandhawitLinSasahira_gluing}.
Combining these slight generalizations together, one sees that the assertion of the lemma holds also for $b_1(M)>0$.
\end{proof}

\begin{remark}\label{rmk: weighted harmonic} Most of the proofs in \cite{Iida,IidaTaniguchi} can be easily adapted to the case $b_{1}(M)>0$, with one exception: the statement and proof of \cite[Claim 3.6]{Iida} needs some non-trivial adjustments when $b_{1}(M)>0$. Consider the non-compact manifold 
\[
M^{*}=M\cup_{\{0\}\times Y}([0,+\infty)\times Y)
\]
with the metric that equals the conical metric $dt\otimes dt+ t^2 dY$ on the end $[0,+\infty)\times Y$. We pick a smooth function $\rho: M^*\to [0,+\infty)$ such that $\rho|_{M}=0$ and $\rho(t,y)=t$ for any $t\geq 1$. For any $k\geq 3$ and any $\alpha\geq 0$, one defines the weighted Sobolev spaces $L^{2,\alpha}_{k}(M,T^*M)$ by completion with respect to the norm 
$\|a\|_{L^{2,\alpha}_{k}}:=\|e^{\alpha \rho}\cdot a\|_{L^{2}_{k}}$.
The spaces $L^{2,\alpha}_{k-1}(M,\mathbb{R})$ and $L^{2,\alpha}_{k-1}(M,\Lambda^{+}_{2}T^*M)$ are defined similarly. Consider the differential operator 
\[
D^{\alpha}:=(d^{*,\alpha},d^+):L^{2}_{k}(M,T^*M)\to L^{2,\alpha}_{k-1}(M,\mathbb{R})\oplus L^{2,\alpha}_{k-1}(M,\Lambda^{+}_{2}T^*M),
\]
where $d^{*,\alpha}=e^{-2\alpha \rho}\circ d^*\circ e^{2\alpha \rho}$. 
In our context, we need to prove the isomorphisms
\begin{equation}\label{eq: weighted harmonic}
\ker D^{\alpha}\cong H^{1}(M;\mathbb{R}),\quad \operatorname{coker}D^{\alpha}\cong H^{2}_{+}(M;\mathbb{R}) 
\end{equation}
for small $\alpha\geq 0$. When $\alpha=0$, this follows from \cite[Theorem 1A]{Hausel04}. Consider the commutative diagram 
\[
\xymatrix{
L^{2}_{k}(M,T^*M) \ar[d]_{\cong}^{e^{\alpha \rho}\cdot-} \ar[rr]^-{D^{\alpha}} & &L^{2,\alpha}_{k-1}(M,\mathbb{R})\oplus L^{2,\alpha}_{k-1}(M,\Lambda^{+}_{2}T^*M)  \ar[d]_{\cong}^{e^{\alpha \rho}\cdot-}\\
L^{2}_{k}(M,T^*M) \ar[rr]^-{D^0+\alpha\cdot Q} & & L^{2}_{k-1}(M,\mathbb{R})\oplus L^{2}_{k-1}(M,\Lambda^{+}_{2}T^*M)}
\]
Here $Q$ is the bounded operator defined by 
\[
Q(a)=(-a^{*}(\rho),(d\rho\wedge a)^{+}),
\]
where $a^*$ is the vector field dual to $a$. Thus, when $\alpha$ is small, $D^{\alpha}$ is a Fredholm operator with the same index as $D^0$. The cokernel of $D^0$ is spanned by $\mathcal{H}^{2}_{+}$, the space of square integrable self-dual harmonic forms on $M^{*}$. We have a well-defined bilinear map 
\[
L:\mathcal{H}^{+}\times (L^{2,\alpha}_{k-1}(M,\mathbb{R})\oplus L^{2,\alpha}_{k-1}(M,\Lambda^{+}_{2}T^*M))\to \mathbb{R},\quad (h,c,b)\mapsto \int_{M^*}h\wedge b.
\]
Note that $L$ is trivial when restricted to $\mathcal{H}^{+}\times\operatorname{image}(D^{\alpha})$. So $L$ decends to a bilinear map 
\[
L': \mathcal{H}^{+}\times\coker D^{\alpha}\to \mathbb{R}.
\]
On the other hand, for any $h\neq 0$, there exists $b$ such that $L(h,0,b)\neq 0$. So $L'|_{\{h\}\times \coker D^{\alpha}}$ is nontrivial. This implies 
\[
\operatorname{dim}(\operatorname{coker}D^{\alpha})\geq \operatorname{dim}(\mathcal{H}^{+})=\operatorname{dim}(\operatorname{coker}D^{0}).\] 
On the other hand, the quantity 
\[
\operatorname{dim}(\operatorname{coker}D^{\alpha})=\operatorname{dim}(\ker D^{\alpha})-\operatorname{index}(D^{0}) 
\]
is upper semi-continuous with respect to $\alpha$. So for $\alpha$ small enough, we have 
\[
\operatorname{dim}(\operatorname{coker}D^{\alpha})= \operatorname{dim}(\mathcal{H}^{+}).
\]
This implies the isomorphism (\ref{eq: weighted harmonic}) for small $\alpha>0$.
\end{remark}

Now we are ready to state and prove two key properties of the reduced contact invariant.

\begin{proposition}\label{pro: contact element S-nilpotent}
There exists $N\gg 0$ such that $S^{N}\psi_{G,\reduced}(Y,\xi)=0$.    
\end{proposition}
\begin{proof}
By the exact triangle (\ref{eq: exact triangle 3}), we have $U\cdot \psi(Y,\xi)=0$. Hence $U\psi_{G,\reduced}(Y,\xi)=0$. Then we apply Corollary \ref{cor: kerU annilated by S} to finish the proof.    
\end{proof}

\begin{proposition}
Assume $b^{+}(M)>0$. Then the image of $\psi_{\reduced}(Y,\xi)$ under the mixed map (\ref{eq: mixed HM}) satisfies 
\begin{equation}\label{eq: contact invariant nontrivial under mixed map}
\langle \overrightarrow{HM}^{*}(M,\fs_{M})(\psi_{\reduced}(Y,\xi)),[\operatorname{Pic}(M,\fs_{M})]\rangle=\pm 1\in \F[U,U^{-1}]/(U).    
\end{equation}
\end{proposition}
\begin{proof} Consider the induced maps
\[
\Psi(Y,\xi)^{*}: H^{*}(S^{b_{1}(M)};\F)\to \widetilde{HM}^*(M,\fs_{M})
\]
and
\[
\widetilde{HM}^*(M,\fs_{M}): \widetilde{HM}^{*}(Y,\fs_{Y};\F)\to H^{*}(\operatorname{Pic}(M,\fs_{M});\F). 
\]
By definition, $\Psi(Y,\xi)^{*}$ sends the generator of $H^{b_{1}(M)}(S^{b_{1}};\F)$ to $\widetilde{\psi}(Y,\xi)$. So by Lemma \ref{lem: mapping degree}, we have 
\[
\langle \widetilde{HM}^*(M,\fs_{M})(\widetilde{\psi}(Y,\xi)),[\Pic(M,\fs_{M})]\rangle=\pm 1.
\]
By the natural transformation form $\tilde{H}^{*}(-)$ to $c\tilde{H}^{*}_{S^1}(-;\F)$, we have the commutative diagram
\[
\xymatrix{\widetilde{HM}^*(Y,\fs_{Y};\F)\ar[d]^{p}\ar[rrr]^-{\widetilde{HM}^*(M,\fs_{M})}& & &H^*(\Pic(M,\fs_{M});\F)\ar[d]^{\operatorname{Id}\otimes 1}\\
\widehat{HM}^*(Y,\fs_{Y};\F)\ar[rrr]^-{\widehat{HM}^*(M,\fs_{M})}& & &H^*(\Pic(M,\fs_{M});\F)\otimes_{\F}\F[U,U^{-1}]/(U)}
\]
Hence 
\[
\begin{split}
\langle \overrightarrow{HM}^{*}(M,\fs_{M})(\psi_{\reduced}(Y,\xi)),[\operatorname{Pic}(M,\fs_{M})]\rangle&= \langle \widehat{HM}^{*}(M,\fs_{M})(\psi(Y,\xi)),[\operatorname{Pic}(M,\fs_{M})]\rangle\\&=\langle \widetilde{HM}^{*}(M,\fs_{M})(\widetilde{\psi}(Y,\xi)),[\operatorname{Pic}(M,\fs_{M})]\rangle\\
&=\pm 1.
\end{split}
\]
\end{proof}

\section{Proof of the main theorems} The purpose of this section is to prove Theorem \ref{theorem:singularities} and Theorem \ref{thm: main}. We will need to use the following key result: 

\begin{proposition}[\cite{KangParkTaniguchi}]\label{prop: homotopy coherent action} Let $\tau_{M}$ be the boundary Dehn twist on a smooth manifold $M$. Suppose $\tau^{m}_{M}=1\in\MCG(M)$ for some $m\neq0$. Let $p$ be a number that is coprime to $m$ and let $G=\mathbb{Z}/p$. Then there exists a homotopy coherent $G$-action on $M$ with trivial cohomological monodromy that extends the Seifert $G$-action on $Y$. Concretely, this means that there exists a smooth fiber bundle 
\begin{equation}\label{eq: Borel M family}
 M\hookrightarrow \widetilde{M}_{\infty}\to BG   
\end{equation}
that has the following properties:
\begin{enumerate}
    \item The monodromy action of $\pi_{1}(BG)$ on $H^{*}(M;\mathbb{Z})$ is trivial.
    \item The family restricts to the family 
    \begin{equation}\label{eq: Borel Y family}
    Y\hookrightarrow \widetilde{Y}_{\infty}\to BG,    
    \end{equation}
    where $\widetilde{Y}_{\infty}=(Y\times EG)/G$.
\end{enumerate}
\end{proposition}

\begin{proof}[Proof of Theorem \ref{thm: main}]
We let $\fs_{M}$ be the canonical spin-c structure on $(M,\omega)$ and let $\fs_{Y}=\fs_{M}|_{Y}$. For the sake of contradiction, we now assume $\tau^m_{M}=0\in \MCG(M)$ for some $m\neq 0$. Let $p$ be a prime number that does not divide $m$. We also assume $p$ is large enough such that the Seifert $G=\mathbb{Z}/p$ action on $Y$ is free and that Lemma \ref{lem: S-localization} applies. We apply Proposition \ref{prop: homotopy coherent action} and obtain the family (\ref{eq: Borel M family}). 
Consider the contact element $\psi_{G,\reduced}(\xi)\in HM^*_{G,\reduced}(Y,\fs_{Y};\F)$. 
By Proposition \ref{pro: contact element S-nilpotent}, we have 
$S^{N}\cdot \psi_{G,\reduced}(Y,\xi)=0$ for some $N$. We take any even $n>2N$ and let $B$ be the $n$-skeleton of $BG$. By pulling back (\ref{eq: Borel M family}), we obtain the family   
\[
M\hookrightarrow \widetilde{M}\to B 
\]
Consider the cobordism induced map 
\[
\overrightarrow{HM}^{*}(\widetilde{M}/B,\fs_{M}): HM^{*}_{G,\reduced}(Y,\fs_{Y};\F)\to H^*(\Pic(\widetilde{M}/B,\fs_{M});\F)\otimes_{\F}\F[U,U^{-1}]/(U).
\]
Then 
\[
\overrightarrow{HM}^{*}(\widetilde{M}/B,\fs_{M})(\psi_{G,\reduced}(Y,\xi)))=\sum_{i\geq 0 }\psi_{i}\otimes U^{-i}
\]
for some $\psi_{i}\in H^*(\Pic(\widetilde{M}/B,\fs_{M});\F)$. By the commutative diagram (\ref{diag: HM-arrow natural 2}) and (\ref{eq: contact invariant nontrivial under mixed map}), we have 
\begin{equation}\label{eq: nontrivial pairing with Pic}
\langle j^{*}(\psi_{0}), [\Pic(M,\fs_{M})]\rangle
    =\langle \overrightarrow{HM}^{*}(\psi_{\reduced}(\xi)), [\Pic(M,\fs_{M})]\rangle
    =\pm 1.    
\end{equation}

Consider the $\F$-coefficient cohomological Serre spectral sequence for the fibration 
\[
\Pic(M,\fs_{M})\to \Pic(\widetilde{M}/B,\fs_{M})\to B.
\]
If there is a nontrivial differential \[d^{0,1}_k: \F^{b_{1}(M)}\cong H^{1}(\Pic(M,\fs_{M});\F)\to H^{2}(B;\F)\cong \F\]
for some $k\geq 2$, then by the Leibniz rule, the map 
\[
j^{*}: H^{b_{1}(M)}(\Pic(\widetilde{M}/B,\fs_{M});\F)\to H^{b_{1}(M)}(\Pic(M,\fs_{M});\F)
\]
must be trivial. This contradicts with (\ref{eq: nontrivial pairing with Pic}). So the Serre spectral sequence collapses on the $E_2$-page. By the Leray--Hirsch theorem, $\psi_{0}$ generates a free sub-$H^{*}(B;\F)$-module of $H^*(\Pic(\widetilde{M}/B,\fs_{M});\F)$. Since 
\[
H^{*}(B;\F)\cong \F[R,S]/(R^2,S^{\frac{n}{2}}),
\]
and since $N<\frac{n}{2}$, we have  \[S^{N}\cdot \psi_0\neq 0\in H^*(\Pic(\widetilde{M}/B,\fs_{M});\F).\]
This implies 
\[
S^{N}\cdot \overrightarrow{HM}^*(\widetilde{M}/B,\fs_{M})(\psi_{G,\reduced}(Y,\xi)))\neq 0\in H^*(\Pic(\widetilde{M}/B,\fs_{M});\F)\otimes_{\F}\F[U,U^{-1}]/(U)
\]
because its leading term $S^{N}\cdot \psi_{0}$ is nonzero.
This is a contradiction since $S^N\cdot \psi_{G,\reduced}(Y,\xi)=0$ and $\overrightarrow{HM}^*(\widetilde{M}/B,\fs_{M})$ is a map of $\F[S,R,U]/(R^2)$-modules. The proof of Theorem \ref{thm: main} is finished.
\end{proof} 

To prove Theorem \ref{theorem:singularities}, we need the following proposition.

\begin{lemma}\label{lem: infinite order when b1>0} Let $Y$ be an oriented Seifert manifold whose base orbifold $C$ has positive genus. Let $M$ be any compact, oriented, smooth 4-manifold bounded by $Y$. Suppose $b_{1}(M)=0$. Then the boundary Dehn twist $\tau_{M}$ has infinite order in $\pi_0(\diff(M,\partial))$. 
\end{lemma}
\begin{proof}
We pick a simple closed loop $\gamma_0$ in $C$ that misses the orbifold points and represents a non-trivial element in $H_{1}(|C|;\mathbb{Q})$. Let $\gamma_{1}$ be a simply closed curve in $Y$ that lifts $\gamma_0$. Since $b_{1}(M)=0$, we have 
\[n[\gamma_{1}]=0\in H_{1}(M;\mathbb{Z})\] for some $n\neq 0$. Thus, letting $\gamma_{2}\hookrightarrow Y$ be $n$ parallel copies of $\gamma_{1}$, there exists a smoothly immersed surface $S\hookrightarrow M$ with $\partial S=\gamma_2$. That implies $[S]$ is nonzero in $H_{2}(M,\gamma_{2};\mathbb{Q})$. Let $F\subset Y$ be the union of all fibers over $\gamma_{1}$, which is an embedded surface representing a non-zero element in $H_{1}(Y;\mathbb{Q})$. Since the maps $H_{2}(Y;\mathbb{Q})\to H_{2}(M;\mathbb{Q})$ and 
$H_{2}(M;\mathbb{Q})\to H_{2}(M;\gamma_{2};\mathbb{Q})$
are both injective, we have 
\[
[F]\neq 0\in H_{2}(M;\gamma_{2};\mathbb{Q}).
\]
By the explicit description of the boundary Dehn twist $\tau_{M}$, it is straightforward to verify that 
\[
\tau_{M,*}([S])=[S]+n[F]\in H_{2}(M;\gamma_{2};\mathbb{Q}).
\]
Therefore, the induced map $\tau_{M,*}: H_{2}(M;\gamma_{2};\mathbb{Q})\to H_{2}(M;\gamma_{2};\mathbb{Q})$ has infinite order. Hence $\tau_{M}$ has infinite order in $\pi_0(\diff(M,\partial))$ as well.
\end{proof}
\begin{proof}[Proof of Theorem \ref{theorem:singularities}] When $b_{1}(Y)=0$, the result follows from Theorem \ref{thm: main}. So we assume $b_{1}(Y)>0$. Since $Y$ is the link of an isolated singularity, the degree of $Y$ (as an orbifold $S^1$-bundle) is negative \cite{neumann-raymond}[\S 5]. Hence the base orbifold $C$ of $Y$ has positive genus. In addition, $b_1(M) = 0$ \cite{milnor}. Thus, the conclusion follows from Lemma \ref{lem: infinite order when b1>0}.
\end{proof}

The proof of the Theorem \ref{thm:singularities2} is identical to that of Theorem \ref{theorem:singularities}, with the additional note that $b_1(M) = 0$ for Milnor fibers of smoothings of isolated normal surface singularities by a result of Greuel--Steenbrink \cite{greuel-steenbrink}.



\bibliographystyle{alpha}
\bibliography{main.bib}

\end{document}